\documentclass[10pt,reqno]{amsart}
\usepackage{amscd}
\usepackage{amsmath,amsthm,amssymb,amsfonts}
\usepackage{mathrsfs}
\usepackage{graphicx}
\usepackage{fancyhdr}
\usepackage{stmaryrd}
\usepackage[colorinlistoftodos,prependcaption,textsize=tiny]{todonotes}
\usepackage{color}
\usepackage{amsmath}
\usepackage{amsfonts}
\usepackage{amssymb}
\usepackage{geometry}
\usepackage{mathtools}
\usepackage{chngcntr}
\usepackage{amsfonts,amsmath,latexsym,verbatim,amscd,mathrsfs,color,array}
\usepackage[colorlinks=true,linkcolor=blue,citecolor=blue]{hyperref}

\usepackage{amssymb}
\usepackage{pdfsync}
\usepackage{epstopdf}

\long\def\hide#1{}

\newcommand{\dif}{\,\mathrm{d}}

\newcommand{\inn}{{\quad\hbox{in } }}

\newcommand{\C}{{\mathbb C}}
\newcommand{\R} {\mathbb R}

\newcommand{\cuad}{{\sqcap\kern-.68em\sqcup}}

\renewcommand{\a}{{\alpha}}

\newcommand{\be}{\begin{equation}}
\newcommand{\ee}{\end{equation}}

\newtheorem{definition}{Definition}
\newtheorem{lemma}{Lemma}[section]
\newtheorem{proposition}{Proposition}[section]
\newtheorem{theorem}{Theorem}

\newtheorem{remark}{Remark}[section]
\newcommand{\bremark}{\begin{remark} \em}
\newcommand{\eremark}{\end{remark} }

\numberwithin{equation}{section}

\numberwithin{equation}{section}

\def \d{\delta}
\def\R{\mathbb{R}}
\def\C{\mathcal{C}}

\def \p{\partial}

\def \e {\varepsilon}

\def \a {\alpha}
\def \RE {\mathop{\text{Re}}}
\def \IM {\mathop{\text{Im}}}
\def \O {\Omega}
\def \L {\mathcal{L}}

\def \d {\tilde{d}}
\def \r {\rho}
\def \v {\theta}

\def \d {\hat{d}}

\def \z {\tilde{z}}

\makeatletter
\def\@seccntformat#1{\@ifundefined{#1@cntformat}%
   {\csname the#1\endcsname\quad}
   {\csname #1@cntformat\endcsname}
}
\makeatother

\author{ Juan D\'avila, Manuel del Pino, Maria Medina and R\'emy Rodiac}
\title[ Interacting helical traveling waves ]{ Interacting helical traveling waves for the Gross-Pitaevskii equation}
\date{\today}

\address[Juan D\'avila]{Department of Mathematical Sciences University of Bath, Bath BA2 7AY, United Kingdom}
\email{jddb22@bath.ac.uk}

\address[Manuel del Pino]{Department of Mathematical Sciences University of Bath, Bat
h BA2 7AY, United
Kingdom, Departamento de Ingenier\'ia Matem\'atica-CMM Universidad de Chile, Santiago 837-
0456, Chile, Chile}
\email{m.delpino@bath.ac.uk}

\address[Mar\'ia Medina]{Departamento de Matem\'aticas,
Universidad Aut\'onoma de Madrid,
Ciudad Universitaria de Cantoblanco,
28049 Madrid, Spain}
\email{maria.medina@uam.es}

\address[R\'emy Rodiac]{Universit\'e Paris-Saclay, CNRS,  Laboratoire de math\'ematiques d'Orsay, 91405, Orsay, France}
\email{remy.rodiac@universite-paris-saclay.fr}

\begin{document}

\begin{abstract}
We consider the 3D Gross-Pitaevskii equation
\begin{equation}\nonumber
i\p_t \psi +\Delta \psi+(1-|\psi|^2)\psi=0 \text{ for } \psi:\R\times \R^3 \rightarrow \mathbb{C}
\end{equation}
and construct traveling waves solutions to this equation. These are solutions of the form \(\psi(t,x)=u(x_1,x_2,x_3-Ct)\) with a velocity \(C\) of order \(\e|\log\e|\) for a small parameter \(\e>0\). We build two different types of solutions. For the first type, the functions \(u\) have a zero-set (vortex set) close to an union of \(n\) helices for \(n\geq 2\) and near these helices \(u\) has degree \(1\). For the second type, the functions \(u\) have a vortex filament of degree \(-1\) near the vertical axis \(e_3\) and \(n\geq 4\) vortex filaments of degree \(+1\) near helices whose axis is \(e_3\). In both cases the helices are at a distance of order \(1/(\e\sqrt{|\log \e|)}\) from the axis and are solutions to the Klein-Majda-Damodaran system, supposed to describe the evolution of nearly parallel vortex filaments in ideal fluids. Analogous solutions have been constructed recently by the authors for the stationary Gross-Pitaevskii equation, namely the Ginzburg-Landau equation. To prove the existence of these solutions we use the Lyapunov-Schmidt method and a subtle separation between even and odd Fourier modes of the error of a suitable approximation.
\end{abstract}

\maketitle
\section{Introduction}

The aim of this paper is to construct solutions to the Gross-Pitaevskii equation
\begin{equation}\label{GP}
i\p_t\psi +\Delta \psi+(1-|\psi|^2)\psi=0 \text{ in } \R \times \R^3,
\end{equation}
for \(\psi: \R \times \R^3 \rightarrow \mathbb{C}\). This equation appears in Bose-Einstein condensates theory, nonlinear optics and superfluidity. At least formally, it possesses two important conserved quantities: the energy 
\begin{equation*}
E(\psi)=\frac12 \int_{\R^3}\left[ |\nabla \psi(\cdot,t)|^2+\frac12 (1-|\psi(\cdot,t)|^2)^2\right] dx,
\end{equation*}
and the momentum
\begin{equation*}
P(\psi)=\int_{\R^3} (i\psi,\nabla \psi) dx,
\end{equation*}
where \((\cdot,\cdot)\) denotes the scalar product in \(\R^2 \simeq \mathbb{C}\).
In this paper we are interested in special solutions called traveling waves solutions. They take the form
\begin{equation}\label{TV}
\psi(t,x)=u(x_1,x_2,x_3-Ct),
\end{equation}
where \(u:\R^3\rightarrow \R\), \(C\in \R\) is a constant to be determined and \(x=(x_1,x_2,x_3)\in \R^3\). If \(\psi\) is defined by \eqref{TV} and solves \eqref{GP} then \(u\) satisfies
\begin{equation}\label{GP1}
iC\p_{x_3} u=\Delta u+(1-|u|^2)u \text{ in } \R^3.
\end{equation}

Traveling waves solutions to \eqref{GP} of finite energy are thought to play an important role in the long time behaviour of solutions, see e.g.\ \cite{Jones_Roberts_1982,Jones_Putterman_Roberts_1986}.
The equation \eqref{GP} is well-posed in various spaces, \cite{Zhidkov_2001,Bethuel_Saut_1999,Gallo_2004,Goubet_2007,Gallo_2008}  and in particular we remark that solutions to \eqref{GP} exist for all time for initial data in the energy space \cite{Gerard_2006,Gerard_2008}. In this article we will construct infinite energy solutions. To find solutions to \eqref{GP1} it is convenient to introduce a small parameter \(\e>0\) and use the scaling \(u_\e(x)=u\left( \frac{x}{\e}\right)\). We are interested in solutions with small velocity, namely we expect the velocity to be of order \(C=C_\e \approx \e|\log \e|\) and thus we set
\begin{equation}\label{eq:speed}
C=C_\e:=c \e|\log\e|
\end{equation}
for a fixed \(c \in \R\). Hence \(u_\e\) is a solution to
\begin{equation}\label{GPeps}
ic \e^2|\log \e| \p_{x_3} u_\e=\e^2 \Delta u_\e+(1-|u_\e|^2)u_\e \text{ in } \R^3.
\end{equation}

The motivation for constructing our solutions originates in the study of the following scaled Gross-Pitaevskii equation
\begin{equation}\label{GPscaled}
i\e^2|\log \e|\p_t \psi+\e^2\Delta \psi+(1-|\psi|^2)\psi=0 \text{ in } \R\times \O,
\end{equation}
where \(\O\) is an open subset of \(\R^3\). Roughly speaking, for initial data whose Jacobian concentrate near some 1D-curve as \(\e\rightarrow 0\), the solution \(\psi(t,\cdot)\) will also concentrate near some 1D-curve that will evolve through the \textit{binormal curvature flow}, see e.g.\ \cite{Jerrard_2002,Jerrard_Smets_2018}. For smooth curves parametrized by arclength \(\gamma(t,s)\), the evolution through binormal curvature flow can be written as
\begin{equation*}
\p_t\gamma=\p_s \gamma \wedge \p_{ss}^2 \gamma.
\end{equation*}
For less regular curves, one can also interpret this flow in a weak sense (see \cite{Jerrard_Smets_2015}). Special solutions to the binormal curvature flow are: a straight line not depending on time, a translating circle and a translating-rotating helix. In each of these examples there exists an associated family of solutions to \eqref{GP}. For the stationary straight line, the associated solution is the standard Ginzburg-Landau vortex of degree \(1\) in the plane, i.e., the solution to
\begin{equation}\label{GL2}
\Delta w+(1-|w |^2)w=0 \text{ in } \R^2,
\end{equation}
which can be written as \(w(z)=\rho(r)e^{i\theta}\) for some non-negative real function \(\rho\)  with \(\rho(0)=0\) and \(\rho(+\infty)=1\). This can be viewed as a stationary solution to \eqref{GP} in \(\R^3\) which is independent of the variable \(x_3\).  We refer to \cite{HerveHerve1994,ChenElliottQi1994,Brezis_Merle_Riviere_1994} for more information on $w$ and to \cite{Mironescu1996,Shafrir1994,Sandier1998}  for its uniqueness properties. Solutions to \eqref{GP} associated to a translating circle are traveling waves solutions with small speed \eqref{eq:speed} exhibiting a vortex ring. They are finite energy solutions and were constructed by variational methods in \cite{Bethuel_Orlandi_Smets_2004,Chiron_2004}. Later on, traveling waves solutions to a similar equation with a vortex ring, the Schr\"odinger map equation, were constructed by a perturbation method in \cite{LinWei2010}, see also \cite{LinWeiJun2013,WeiJun2012}.  We refer to \cite{Bethuel_Gravejat_Saut_2008,Bethuel_Gravejat_Saut_2009,Bethuel_Gravejat_Saut_2009b, Maris_2013,Chiron_Maris_2017,Liu_Wei_2020,Bellazzini_Ruiz_2020,Chiron_Pacherie_2020_a,Chiron_Pacherie_2020_b} for more on finite energy solutions.
Associated to the helix, there exist infinite energy solutions to \eqref{GP}. They are also traveling waves solutions with small speed and were constructed by Chiron in \cite{Chiron_2005}, by using variational methods. The corresponding traveling single-helix solutions to the Schr\"odinger map equation were proved to exist by Wei-Yang in \cite{WeiJun2016}, where the authors raise the open problem of the existence of solutions with a vortex-set of multiple helices. One of the purposes of this article is to answer this question for the Gross-Pitaevskii equation.

Once we know that the straight filament is a solution to \eqref{GP}, one can also look for solutions to the GP-equation whose vortex set consists in multiple, almost straight, parallel filaments. In this case, it is believed that the motion of these \(n \geq 2\) filaments is governed by the Klein-Majda-Damodaran system:
\begin{equation}\label{eq:KMD}
-i\p_t f_k(t,z)-\p_{zz} f_k(t,z) -2\sum_{j\neq k} d_jd_k \frac{f_k-f_j}{|f_k-f_j|^2}=0, \quad k=1,\dots,n,
\end{equation} 
where \(z\) is the third coordinate in \(\R^3\) and \(d_i\), \(i=1,\dots,n\), represent the topological degree around the filament.
This system was derived in \cite{Klein_Majda_Damodaran1995} in the context of fluid mechanics and studied in \cite{Kenig_Ponce_Vega2003}. The Euler equation and the Gross-Pitaevskii equation are thought to share many common properties, in particular with respect to the behavior of their vortex filaments. Recently, Jerrard-Smets in \cite{Jerrard_Smets_2020} provided the first rigorous justification of the appearance of the Klein-Majda-Damodaran system as a limiting problem for vortex filaments of the Gross-Pitaevskii solutions. More precisely, they proved that for well-prepared initial data, the vortex set of solutions to \eqref{GP} converges, as \(\e\rightarrow 0\), towards \(n\) almost parallel filaments solutions to the Klein-Majda-Damodaran system. In this work only degree \(d_i=+1\) were considered.
This follows an earlier work on the interaction of vortex filaments for the Ginzburg-Landau equation by Contreras-Jerrard \cite{ContrerasJerrard2017}. 

The result in 
\cite{Jerrard_Smets_2020} is based on  variational arguments, and therefore only finite energy solutions are considered in cylindrical domains of the form $\omega\times \R$ where $\omega \subset \R^2$ is bounded, with periodicity in the third variable. The finite energy condition is not a natural hypothesis for nearly parallel vortex filaments in the entire space, since the standard Ginzburg-Landau vortex of degree 1 has infinite energy in $\R^2$. 

\medskip
In this paper we consider an important family of explicit solutions of system \eqref{eq:KMD} given by rotating and translating helices of degree one.
More precisely, for $n\geq 2$ we consider the solution to \eqref{eq:KMD} given by 
\begin{equation}
\label{fk}
f_k(t,z):=\hat{d}e^{i(z-\nu t)}e^{\frac{2i(k-1)\pi}{n}}, \quad k=1,\dots,n,\; \text{ for } \hat{d}:=\sqrt{\frac{n-1}{1-\nu}},\mbox{ where } \nu<1.
\end{equation}
The curves in $\R^3$ described by $z \mapsto (f_k(t,z),z)$ are helices arranged with polygonal symmetry. 

\medskip 
Our goal in this paper is to construct a family of solutions to \eqref{GPeps}  whose vortex-set is close, as \(\e\) tends to zero, to the helices \eqref{fk}. 
The solutions we construct look like a product of standard vortices of degree one \(w\), i.e., the solution to \eqref{GL2}, centered at \(f_k(t,z)\), in the planes perpendicular to the vertical axis \(e_3\).
The solutions we construct are periodic in $t$ and $z$, just as the helices \eqref{fk}. In addition we obtain a refined asymptotic description of the solution, not available in \cite{Jerrard_Smets_2020}.

We denote by $(r,\theta,x_3)$ the usual cylindrical coordinates.

\begin{theorem}\label{th:main1}
For each \(n\geq 2\) and for every \(-\infty<c<1\), there exists \(\e_0>0\) such that for every \(0<\e<\e_0\)  there exists \(u_\e\) which solves \eqref{GP1} with \(C=c\e|\log \e|\). The solution \(u_\e\) can be written as
\begin{equation*}\begin{split}
u_\e(r,\theta,x_3)& =\prod_{k=1}^{n} w\left(re^{i\theta}-d_\e e^{i\e x_3} e^{2ik\pi/n} \right)+\varphi_\e \\
\end{split}
\end{equation*}
with \begin{equation*}
\|\varphi_\e\|_{L^\infty}\leq \frac{M}{|\log \e|} \text{ for some constant } M>0,
\end{equation*}
and \(d_\e:= \frac{\d_\e}{\e\sqrt{|\log\e|}}\) with \(\hat{d}_\e =\sqrt{\frac{n-1}{1-c}} +o_\e(1)\).
\end{theorem}

\begin{remark}
The corresponding solutions to \eqref{GP} given by Theorem \ref{th:main1} are
\begin{equation}\nonumber
\psi_\e(t,x)=\prod_{k=1}^n w\left( re^{i\theta} -d_\e e^{i\e(x_3-c\e|\log \e|t)}e^{2ik\pi/n} \right)+\varphi_\e(x_1,x_2,x_3-c\e|\log\e|t).
\end{equation}
Furthermore, thanks to the symmetries of equation \eqref{GP1} we can see that for all constant \( +\infty >c> -1\) there exists a solution \(\tilde{u}_\e\) to \eqref{GP1} with \(C=-c\e |\log \e| \) if \(\e\) is small enough. This solution can be written as 
\begin{equation*}\begin{split}
\tilde{u}_\e(r,\theta,x_3)& =\prod_{k=1}^{n} w\left(re^{i\theta}-d_\e e^{-i\e x_3} e^{2ik\pi/n} \right)+\tilde{\varphi}_\e 
\end{split}
\end{equation*}
with \begin{equation*}
\|\tilde{\varphi}_\e\|_{L^\infty}\leq \frac{M}{|\log \e|} \text{ for some constant } M>0,
\end{equation*}
and \(d_\e:= \frac{\d_\e}{\e\sqrt{|\log\e|}}\) with \(\hat{d}_\e =\sqrt{\frac{n-1}{1+c}} +o_\e(1)\).
\end{remark}

Our result extends the pioneering work of Chiron \cite{Chiron_2005} to the case of 2 or more interacting helical filaments.
In \cite{Chiron_2005} a solution with a single helicoildal vortex filament was built by a subtle constrained minimization procedure.

\medskip

We can also consider the solution to \eqref{eq:KMD} which consists in $n+1$, with $n\geq 3$, helices of degree one rotating around a straight filament of degree \(-1\):
\begin{equation*}
d_0=-1, \quad d_k=+1, \quad k=1,\dots, n+1,
\end{equation*}
\begin{equation*}
f_0(t,z)=z, \quad  f_k(t,z)=\hat{d}e^{i(z-\nu t)} e^{\frac{2i(k-1)\pi}{n+1}}, \quad k=1,\dots,n+1,\; \text{ for } \hat{d}:=\sqrt{\frac{n-2}{1-\nu}},\mbox{ where }\nu<1.
\end{equation*}

\begin{theorem}\label{th:main2}
For each \(n\geq 3\) and for \(-\infty<c<1\), there exists \(\e_0>0\) such that for every \(0<\e<\e_0\)  there exists \(u_\e\) which solves \eqref{GP1} with \(C=c\e|\log \e|\). The solution \(u_\e\) can be written as
\begin{equation}
\begin{split}
u_\e(r,\theta,x_3)&=\overline{w}(re^{i\theta})\prod_{k=1}^{n+1} w\left(re^{i\theta}-d_\e e^{i\e x_3}e^{2i(k-1)\pi/(n+1)} \right)+\varphi_\e \\
\end{split}
\end{equation}
with \begin{equation}
\|\varphi_\e\|_{L^\infty}\leq \frac{M}{|\log \e|} \text{ for some constant } M>0,
\end{equation}
and \(d_\e:= \frac{\d_\e}{\e\sqrt{|\log\e|}}\) with \(\hat{d}_\e= \sqrt{\frac{n-2}{1-c}} +o_\e(1)\).
\end{theorem}

Linking with fluid mechanics, we point out that helical solutions to the Euler equations have been built recently in \cite{Davila_delPino_Musso_Wei_2020}. The solutions constructed in Theorem \ref{th:main1} and Theorem \ref{th:main2} are counterpart of solutions with helical interacting vortex filaments constructed in \cite{DdPMR} for the Ginzburg-Landau equation. Indeed, our strategy is to look for a solution of \eqref{GPeps} which, at main order, resembles
\begin{equation}\label{aprox1}
u_d(r,\theta,x_3)=\prod_{k=1}^{n} w\left(\frac{r}{\varepsilon}e^{i\theta}-d_\e e^{i x_3} e^{2ik\pi/n} \right)
\end{equation}
for Theorem \ref{th:main1} and
\begin{equation}\label{aprox2}
u_d(r,\theta,x_3)=\overline{w}\left(\frac{r}{\varepsilon}e^{i\theta}\right)\prod_{k=1}^{n+1} w\left(\frac{r}{\varepsilon}e^{i\theta}-d_\e e^{i x_3}e^{2i(k-1)\pi/(n+1)} \right)
\end{equation}
for Theorem \ref{th:main2}. Although these approximations do not fully possess the helical symmetry, we can show that \(e^{-inx_3}u_d(r,\theta,x_3)\) are screw symmetric. Since the Gross-Pitaevskii equation is invariant by screw symmetry, we can take advantage of this fact to reduce the problem to a 2D problem. 

In order to construct our solutions via a perturbative approach and a Lyapunov-Schmidt argument the strategy is the following: we first compute the error of our approximation, then we develop a linear theory for a suitable projected problem and  we use a fixed point argument. Finally we adjust the parameter \(d_\e\) to find an actual solution to \eqref{GP}. Here a major difficulty appears: the error contains terms of order \(O\left(|\log \e|^{-1}\right)\) which are orthogonal to the kernel of the linearized operator. Hence the vortex-location adjustment, which arise by multiplying the equation by the kernel of the linearized operator and integrating by parts, takes place at order \(O(\e\sqrt{|\log \e|})\). This is much smaller than the size of the non-linear terms which, in concordance with the size of the error, is of order \(O\left(|\log \e|^{-2}\right)\). To be able to conclude we need to use a careful decomposition of the perturbation in ``even'' and ``odd'' Fourier modes and to show that this decomposition is respected by the non-linearity of the equation. The even part of the decomposition will be of order \(O\left(|\log \e|^{-1}\right)\) whereas the odd part will be of order \(O(\e\sqrt{|\log \e|})\). By symmetry, the even part and the nonlinearity applied to this even part are orthogonal to the kernel and thus do not play any role in the reduction argument. The same difficulty arises in \cite{DdPMR}, where, for pedagogical purposes, only the case of two vortices was considered. In the present article we treat in detail the general case of \(n\) vortices.  A novelty of this work compared to \cite{DdPMR} is that the travelling wave effect makes the problem more delicate since the remote regime changes substantially. The derivation of the reduced equations is more subtle for the same reason.  The analogy here discovered may be regarded as a 3-dimensional parallel to that between stationary Ginzburg-Landau vortices and  Gross-Pitaevskii 
``vortex pair'' \cite{Bethuel_Saut_1999,Chiron_Maris_2017,Wei_Liu_2020}  where substantial technical work is needed to handle the travelling wave effect. More precisely, we have to deal with the new term in the Gross-Pitaevski equation \(iC\p_{x_3}u\). We check that the new error created by this term when applied to the ansatz is small enough, has sufficient decay and the same is true for its even and odd parts. We also prove that we can obtain good linear estimates for the linearized Gross-Pitaevskii operator.

The paper is organized as follows, in section \ref{IV} we explain the use of the screw-symmetric invariance of the equation \eqref{GP1} and the approximation to reduce the problem to a 2D problem. Then we look for a solution to \eqref{GP1} under an additive-multiplicative perturbation of our approximation. This is nowadays usual in equations with complex valued unknowns presenting a vortex structure (this special form of the perturbation was first devised in \cite{delPinoKowalczykMusso2006}). In section \ref{V} we compute the error of the approximation and estimate its size and decay properties. We also consider the size of the ``odd'' and ``even'' Fourier modes separately. Section \ref{VI} is devoted to the analysis of the linearized projected problem and the non-linear projected problem. Here we use elliptic estimates and the Fredholm alternative for the linearized problem and the Banach fixed point theorem for the non-linear problem.  In section \ref{VII} we study the reduced problem, i.e.,\ we justify that we can cancel the Lyapunov-Schmidt coefficients arising in the previous section. The reduced problem is solved by a continuity argument.

\section{Formulation of the problem}\label{IV}
\subsection{Reduction to a two dimensional problem.} As a first step to prove our theorems we will reduce the problem to a two-dimensional one by using a screw or helicoidal symmetry. For convenience, we use cylindrical coordinates, i.e., $(r,\theta, x_3)\in \R^+\times\R\times\R$ and we consider \(2\pi\)-periodic functions in \(\theta\).

\begin{definition}\label{screwsim}
We say that a function $u$ is screw-symmetric if
\begin{equation*}
u(r,\theta+h,x_3+h)=u(r,\theta,x_3)
\end{equation*}
for any $h \in\R$. Equivalently
$$u(r,\theta,x_3)=u(r,\theta-x_3,0)=:U(r,\theta-x_3).$$
\end{definition}
Writing the standard vortex of degree one in polar coordinates, i.e., $w(re^{i\theta})= \rho(r)e^{i\theta}$, we can see that the approximations $u_d$ defined in \eqref{aprox1} and \eqref{aprox2} satisfy
\[u_d(r,\theta,x_3)=e^{inx_3}u_d(r,\theta-x_3,0). \]
That is, $u_d$ is not screw-symmetric but $\tilde{u}_d(r,\theta,x_3):=e^{-inx_3}u_d(r,\theta,x_3)$ is, what suggests to look for solutions $u$ of \eqref{GPeps} in the form
\[u(r,\theta,x_3)=e^{inx_3}U(r,\theta-x_3),\]
being $U:\R^+\times \R$ a $2\pi$-periodic function in the second variable.  Denoting $U=U(r,s)$, this corresponds to ask $U$ to be a solution of
\begin{equation*}
\e^2\left(\p^2_{rr}U+\frac{1}{r}\p_r U+\frac{1}{r^2}\p_{s s}^2 U+\p^2_{ss} U-2in\p_sU-n^2U\right) -ic|\log \e|\e^2(inU-\p_sU)+(1-|U|^2)U 
=0,
\end{equation*}
or, in rescaled coordinates, to find a solution $V(r,s):=U(\e r,s)$ to the equation
\begin{multline}\label{eq:GP2Drescaled}
\p^2_{rr}V+\frac{1}{r}\p_r V+\frac{1}{r^2}\p_{s s}^2 V+\e^2(\p^2_{ss} V-2in\p_sV-n^2V)-ic|\log \e|\e^2(inV-\p_sV)+(1-|V|^2)V 
=0
\end{multline}
in $\R^+\times \R$.

From now on we will work in the plane $\R^2$, and we will use the notation $z=x_1+ix_2=re^{is}$. We denote by $\Delta$ the Laplace operator in $2$-dimensions, meaning
\begin{equation*}
\Delta=\p^2_{x_1x_1}+\p^2_{x_2x_2}=\p^2_{rr}+\frac{1}{r}\p_r+\frac{1}{r^2}\p^2_{ss}.
\end{equation*}
\noindent and then equation \eqref{eq:GP2Drescaled} can be written as
\begin{align}
\label{eq:GP2Drescaled-2}
\Delta V
+\e^2(\p^2_{ss} V-2in\p_sV-n^2V)-ic|\log \e|\e^2(inV-\p_sV)+(1-|V|^2)V=0 \text{ in } \R^2.
\end{align}
In the new coordinates we will write the approximation in general form as
\begin{equation}\label{eq:approx}
V_d(z)=\prod_{j=1}^{n^+} w(z-\xi^+_j) \prod_{k=1}^{n^-} \overline{w}(z-\xi^-_k), 
\end{equation}
with \(n=:n^+-n^-\). For Theorem \ref{th:main1} we will take \(n^+=n\), \(n^-=0\) and \(\xi^+_j=d_\e e^{2i\pi (j-1)/n}\) whereas for Theorem \ref{th:main2} we will take \(n^+=n+1\), \(n^-=1\), \(\xi^+_j=d_\e e^{2i\pi (j-1)/(n+1)}\)  and \(\xi^-_1=0\). Here
\begin{equation}\label{eq:relationd}
d_\e:=\frac{\hat{d}_\e}{\e \sqrt{|\log \e|}},
\end{equation}
for some new parameter $\hat d_\e = O(1)$.

\subsection{Additive-multiplicative perturbation}
Let us define the solution operator
\begin{equation}\label{SV}
S(v) := 
\Delta v
+\e^2(\p^2_{ss}v-2ni\p_sv-n^2v)-ic|\log \e|\e^2(inv-\p_sv)+(1-|v|^2)v,
\end{equation}
so that the equation to be solved can be written as 
\begin{align}
\label{mainEq}
S(v)=0.
\end{align}
Recall the notation $z = r e^{is} = x_1 + i x_2$ and $\Delta = \partial^2_{x_1x_1} + \partial^2_{x_2x_2} $.
Notice that when using the coordinates $(x_1,x_2)$ equation \eqref{mainEq} is posed in $\R^2$, while if we use polar coordinates $(r,s)$ the domain for \eqref{mainEq} is $r>0$, $s\in \R$ with periodicity.

Following  del Pino-Kowalczyk-Musso  \cite{delPinoKowalczykMusso2006},
we look for a solution to \eqref{mainEq} of the form
\begin{align}
\label{eq1}
v = \eta V_d (1+i\psi)  + (1-\eta) V_d e^{i\psi} ,
\end{align}
where $V_d$ is the ansatz \eqref{eq:approx} and $\psi$ is the new unknown.
The cut-off function $\eta$ in \eqref{eq1} is defined as
\begin{align}\label{def:cut_off_eta}
\eta(z) = \sum_{j=1}^{n^+} \eta_1(|z-\xi_j^+|)+\sum_{k=1}^{n^-}\eta_1(|z-\xi_k^-|), \quad z\in \mathbb C=\R^ 2,
\end{align}
and  $\eta_1:\R \rightarrow [0,1]$ is a smooth cut-off function such that
\begin{align}
\label{eta1}
 \eta_1(t)=1 \text{ for } t\leq 1\text{ and }\eta_1(t)=0 \text{ for } t\geq 2.
\end{align}
 The reason for the form of the perturbation term in \eqref{eq1} is the same as in \cite{delPinoKowalczykMusso2006}.
On one hand, the nonlinear terms behave better for the norms that we consider when using the multiplicative ansatz, but near the vortices, an additive ansatz is better since it allows the position of the vortex to be adjusted.

We would like to rewrite \eqref{mainEq} into an equation on $\psi$ of the form 
\[
\mathcal L^\varepsilon(\psi) =-E+\mathcal{N}(\psi)
\]
where $\mathcal L^\varepsilon$ is a linear operator, $E$ is the error of the approximation and $\mathcal N(\psi)$ groups the nonlinear terms. However, we expect \(\phi:=iV_d\psi\) to be a smooth function which does not necessarily vanish near the vortices. Hence \(\psi=-i\phi/V_d\) is not a distribution in general (although it is a function in \(C^\infty(\R^2\setminus \{ \xi_j^+,\xi_k^-\})\) it is not a \(L^1_{\text{loc}}(\R^2)\) function). Thus the global problem we want to solve will take a separate form near the vortices and far away from them.  Given two real numbers $a,b$, with $a<b$, we define the set
$$B_a^b:=\big\{\bigcup_{j=1,\ldots, n^+}\{z\in \mathbb{C}:\, a\leq |z-\xi_j^+|\leq b|\}\big\}\cup\big\{\bigcup_{k=1,\ldots, n^-}\{z\in \mathbb{C}:\, a\leq |z-\xi_k^-|\leq b|\}\big\},$$
and $B^b:=B_0^b$.

\begin{lemma}\label{lem:formulation}
Let \(\phi \in \C^\infty(\R^2)\). There exists a small constant $\rho_0>0$ such that, if $\|\phi\|_{L^\infty(\R^2)}<\rho_0$, the function \(v=\eta (V_d+\phi)+(1-\eta)V_de^{\frac{\phi}{V_d}}\) is a solution of \(S(v)=0\), where \(S\) is defined by \eqref{SV} if and only if \(\phi\) satisfies
\begin{equation}
\eta L_0(\phi) +(1-\eta)iV_d L'(\psi)=-E +N(\phi),
\end{equation}
where \(\psi=\frac{\phi}{iV_d}\) and
\begin{align}
L_0(\phi)&:= \Delta \phi+\e^2(\p^2_{ss}\phi-2ni\p_s\phi-n^2\phi)-ic|\log \e|\e^2(in\phi-\p_s\phi) +(1-|V_d|^2)\phi-2\RE(\overline{V_d}\phi)V_d,  \label{L0}\\
L'(\psi)&:=  \Delta \psi +2\frac{\nabla V_d}{V_d}\nabla \psi -2i|V_d|^2\IM(\psi)+\varepsilon^2 \Bigl( \partial_{ss}^2 \psi
+\frac{2\partial_s V_d }{V_d}\partial_s \psi 
- 2in \partial_s \psi  \Bigr) 
+ic|\log \e|\e^2\p_s\psi, \label{eq:mathL}
\end{align}
\begin{equation}\label{eq:def_error}
E:= S(V_d),
\end{equation}
\begin{equation}\label{def:N}
N(\phi):=-(1-\eta)iV_d\left[i(\nabla \psi)^2+i\e^2(\p_s\psi)^2-i|V_d|^2(e^{-2\IM(\psi)}-1+2\IM(\psi))\right]-M(\phi),
\end{equation}
where \(M(\phi)\) is a smooth function of \(\phi\) which is a sum of terms at least quadratic, localized in the area \(\eta \neq 0\). Furthermore, \(M(\phi)\) is a sum of analytic functions of \(\phi\) multiplied by cut-off functions and 
\begin{equation}
|M(\phi) | \leq C \|\phi\|^2_{C^1(B^2)}
\end{equation}
if \( \|\nabla \phi\|_{L^\infty}+\|\phi\|_{L^\infty} \leq C_0\) for \(C_0\) small enough.  At last, if \(\phi=iV_d \psi\)
\begin{equation}\label{eq:rel_L_and_L'}
L_0(\phi)=iV_d L'(\psi)+iE \psi \quad \text{ in } \R^2 \setminus \{\xi_j^+,\xi_k^-, j=1,\dots,n^+,\ k=1,\dots,n^-\}.
\end{equation}
\end{lemma}

\begin{remark} In the lemma above and in its proof below, the function \(\psi=\frac{\phi}{iV_d}\) is used only in the zones where \( (1-\eta)\) does not vanish, i.e., only far from the vortices. In these zones, \(\psi\) is a distribution because \(\phi\) is a distribution in \(\R^2\) by assumption, and \(V_d\) is a smooth function which does not vanish far from the vortices.
\end{remark}
\begin{proof}

We follow \cite[Lemma 2.7]{Chiron_Pacherie_2020_a}. We start by proving \eqref{eq:rel_L_and_L'}. This  can be seen in the following computation, valid in the sense of distributions, far away from the vortices:
\begin{align*}
L_0(iV_d\psi)&= \Delta (iV_d\psi)+(1-|V_d|^2)(iV_d\psi) \\
& \quad+\e^2\left[\p^2_{ss}(iV_d\psi)-2ni\p_s(iV_d\psi)-n^2(iV_d\psi)\right]-ic|\log \e|\e^2\left[in(iV_d\psi)-\p_s(iV_d\psi)\right] \\
& \quad \quad-2\RE(iV_d\psi\overline{V_d})V_d \\
&=i \Bigl[\Delta V_d+ \e^2\left(\p^2_{ss}V_d-2ni\p_s(V_d)-n^2V_d\right)-ic|\log \e|\e^2\left(in V_d-\p_s V_d\right)\Bigr]\psi \\
& \quad +iV_d \Bigl[ \Delta \psi+2\frac{ \nabla V_d}{V_d} \nabla \psi + \e^2\left( \p^2_{ss} \psi +2\frac{\p_s V_d}{V_d}\p_s \psi -2n\p_s \psi\right)+ic|\log \e| \e^2 \p_s \psi \Bigr] \\
& \quad \quad +(1-|V_d|^2)(iV_d\psi)+2|V_d|^2\IM(\psi) V_d \\
&=iE \psi+iV_d L'(\psi).
\end{align*}
Now we decompose
\begin{align*}
S(v)  = S_0(v) + S_1(v),
\end{align*}
with
\begin{align}
S_0(v) := \Delta v + (1-|v|^2) v, \quad
S_1(v) := 
\varepsilon^2(\p^2_{ss}v-2ni\p_sv-n^2v) -ic|\log \e|\e^2(inv-\p_sv). \label{def:S_0S_1}
\end{align}

\noindent For the rest of the proof we set
\begin{equation}\nonumber
\zeta:= V_d(1+i\psi-e^{i\psi}) \text{ in } \{ (1-\eta)\neq 0\}.
\end{equation}
Since \(v= \eta (V_d+\phi)+(1-\eta)V_de^{i\psi}\) with \( \phi=iV_d\psi\), we have

\begin{align*}
\Delta v &= \eta \left( \Delta V_d +\Delta \phi\right) +(1-\eta)\Delta (V_de^{i\psi})+2\nabla \eta  \left[\nabla V_d+i\nabla (V_d \psi)-\nabla (V_de^{i\psi})\right] +\Delta \eta (V_d+iV_d \psi -V_de^{i\psi}) \\
&= \eta \left(\Delta V_d +\Delta \phi\right) +(1-\eta)(\Delta V_de^{i\psi}+V_d\Delta (e^{i\psi})+2\nabla V_d \nabla (e^{i\psi})) +2\nabla \eta \nabla \zeta +\Delta \eta \zeta \\
&= \eta \left(\Delta V_d+ \Delta \phi \right) +(1-\eta)(\Delta V_d e^{i\psi}+V _d(i\Delta \psi-(\nabla \psi)^2)e^{i\psi}+2i\nabla V_d \nabla \psi e^{i\psi})+2\nabla \eta \nabla \zeta +\Delta \eta \zeta.
\end{align*}

\noindent By using that far from the vortices, \(\Delta \phi= \Delta (iV_d\psi)=i\Delta V_d\psi +iV_d \Delta \psi +2i \nabla V_d \nabla \psi\) we can write
\begin{multline}\label{eq:Laplacian_v}
\Delta v =(\eta +(1-\eta)e^{i\psi}) \left( \Delta V_d+ \Delta \phi \right) 
+(1-\eta)e^{i\psi} \left[  -V_d(\nabla \psi)^2-i\Delta V_d \psi \right] +2\nabla \eta \nabla \zeta +\Delta \eta \zeta.
\end{multline}
We then set \( A:=V_d+\phi\) and \(B:=V_de^{i\psi}\) (\(B\) is defined far from the vortices), thus \(v= \eta A+(1-\eta)B\) and
\begin{align*}
(1-|v|^2)v&= (1-|\eta A+(1-\eta)B|^2)(\eta A+(1-\eta)B) \\
&= \left[1-\eta^2|A|^2-(1-\eta)^2|B|^2-2\eta (1-\eta)\RE(A\overline{B})\right](\eta A +(1-\eta)B).
\end{align*}
We want to make the terms \( \eta (1-|A|^2)A+(1-\eta)(1-|B|^2)B\) appear. Hence we write
\begin{equation*}\begin{split}
(1-|v|^2)v=&\,\eta (1-|A|^2)A+\eta A[ (1-\eta^2)|A|^2-(1-\eta)^2|B|^2-2\eta (1-\eta)\RE(A\overline{B})] \\
&+(1-\eta)(1-|B|^2)B +(1-\eta)B [ (1-(1-\eta)^2)|B|^2-\eta^2|A|^2-2\eta(1-\eta)\RE(A\overline{B})].
\end{split}\end{equation*}
We factorize \(\eta(1-\eta)\) and write
\begin{align*}
(1-|v|^2)v&=\eta (1-|A|^2)A+(1-\eta)(1-|B|^2)B \\
&\quad +\eta(1-\eta)\left[(1+\eta)A|A|^2-(1-\eta)A|B|^2-2\eta A\RE(A\overline{B})\right] \\
& \quad+\eta (1-\eta)\left[ (2-\eta)B|B|^2-\eta B|A|^2-2(1-\eta)B\RE(A\overline{B})\right] \\
&=\eta (1-|A|^2)A+(1-\eta)(1-|B|^2)B  \\
& \quad +\eta (1-\eta)\Bigl[A|A|^2+2B|B|^2-A|B|^2-2B\RE(A\overline{B}) \\
& \quad \quad +\eta \left(A|A|^2-B|B|^2+A|B|^2-B|A|^2-2A \RE(A \overline{B} \right) +2B \RE (A \overline{B})\Bigr] \\
&= \eta (1-|A|^2)A+(1-\eta)(1-|B|^2)B \\
& \quad +\eta (1-\eta) \left[ F_1(A,B)+\eta F_2(A,B)\right]
\end{align*}
where \(F_1(A,B), F_2(A,B)\) are real analytic functions of \(A\) and \(B\) and vanish for  \(A=B\). Since, in the zone where \(\eta(1-\eta)\) is nonzero, \(A-B=\zeta\) we can write
\begin{align*}
(1-|v|^2)v=\eta (1-|A|^2)A+(1-\eta)(1-|B|^2)B+\eta (1-\eta)\left[\zeta G_1(\phi)+\overline{ \zeta}H_1(\phi)+\eta (\zeta G_2(\phi)+\overline{\zeta}H_2(\phi) \right]
\end{align*}
where \( G_1,G_2,H_1,H_2\) are  real analytic functions of \(\phi\) satisfying \( |H_{i}(\phi)|, |G_{i}(\phi)| \leq C(1+|\phi|+|e^{\phi}|)\), $i=1,2$, where \(C>0\) is a universal constant. Since \(A=V_d+\phi\) we have
\begin{align}
(1-|A|^2)A &= (1-|V_d+\phi|^2)(V_d+\phi) \nonumber \\
&=(1-|V_d|^2-|\phi|^2-2\RE(\overline{V_d}\phi))(V_d+\phi)  \nonumber\\
&=(1-|V_d|^2)V_d-2\RE(\overline{V_d}\phi) V_d+(1-|V_d|^2)\phi -|\phi|^2(V_d+\phi)-2\RE(\overline{V_d}\phi)\phi. \label{eq:A}
\end{align}

\noindent We also have, when \( (1-\eta)\neq 0\), \( B=V_de^{i\psi}\) and
\begin{align}
(1-|B|^2)B &= (1-|V_de^{i\psi}|^2)V_de^{i\psi} \nonumber \\
&= (1-|V_d|^2e^{-2\psi_2})V_de^{i\psi}\nonumber  \\
&= (1-|V_d|^2)V_de^{i\psi} +2|V_d|^2\IM(\psi)V_de^{i\psi}-|V_d|^2V_de^{i\psi}(e^{-2\IM(\psi)}-1+2\IM(\psi)) \nonumber \\
&=V_de^{i\psi} \left[ (1-|V_d|^2)+2|V_d|^2\IM(\psi)-|V_d|^2(e^{-2\IM(\psi)}-1+2\IM(\psi)) \right]. \label{eq:B}
\end{align}
We use the relations \eqref{eq:A} and \eqref{eq:B}, along with \( 2|V_d|^2\IM(\psi)=-2\RE(\overline{V_d}\phi)\), to obtain
\begin{align}
(1-|v|^2)v &= (\eta +(1-\eta)e^{i\psi}) \left[ (1-|V_d|^2)V_d -2\RE(\overline{V_d}\phi) V_d+(1-|V_d|^2)\phi
\right] \nonumber \\
& -\eta \left( |\phi|^2(V_d+\phi)+2\RE(\overline{V_d}\phi)\phi \right) \nonumber \\
&+(1-\eta)e^{i\psi}\Bigl[(|V_d|^2V_d \left(e^{-2\IM(\psi)}-1+2\IM(\psi) \right)- (1-|V_d|^2)\phi\Bigr] \nonumber\\
&+ \eta(1-\eta)\left[\zeta G_1(\phi)+\overline{ \zeta}H_1(\phi)+\eta (\zeta G_2(\phi)+\overline{\zeta}H_2(\phi) \right]. \label{eq:potential}
\end{align}
We add \eqref{eq:Laplacian_v} and \eqref{eq:potential} to see that
\begin{equation}\begin{split}\label{eq:S_equal_0_int1}
S_0(v)=&\,\left( \eta +(1-\eta)e^{i\psi} \right) \Bigl[ (\Delta V_d+(1-|V_d|^2)V_d) +\Delta \phi -2\RE(\overline{V_d}\phi) V_d+(1-|V_d|^2)\phi  \Bigr]  \\
&-\eta \left( |\phi|^2(V_d+\phi)+2\RE(\overline{V_d}\phi)\phi \right)\\
&+(1-\eta)iV_de^{i\psi} \left[i(\nabla \psi)^2-\frac{\Delta V_d}{V_d}\psi-i|V|^2(e^{-2\IM(\psi)}-1+2\IM(\psi))-(1-|V_d|^2)\psi\right]  \\
&+\eta (1-\eta) \left[\zeta G_1(\phi)+\overline{ \zeta}H_1(\phi)+\eta (\zeta G_2(\phi)+\overline{\zeta}H_2(\phi) \right]+2\nabla \eta \nabla \zeta +\Delta \eta \zeta.
\end{split}\end{equation}
Similarly we compute
\begin{align*}
S_1(v)&=\eta \left[ S_1(V_d)+S_1(\phi)\right]+(1-\eta)S_1(V_de^{i\psi}) +(\e^2 \p^2_{ss} \eta -2\e^2 n i\p_s \eta+ic |\log\e| \e^2\p_s \eta)V_d(i\psi+1-e^{i\psi}) \\
& \quad \quad +2\e^2\p_s\eta\p_s (iV_d\psi +V_d-V_de^{i\psi} ) \\
&=\eta \Bigl[ S_1(V_d)+S_1(\phi) \Bigr] \\
& \quad+(1-\eta)e^{i\psi}\Bigl[ S_1(V_d)+iV_d \bigl(\e^2\p^2_{ss}\psi+2\frac{\p_s V_d}{V_d}\p_s \psi-2in\p_s \psi+ic|\log \e| \e^2 \p_s \psi \bigr) +iV_d \e^2i (\p_s \psi)^2 \Bigr] \\
& \quad \quad +(\e^2 \p^2_{ss} \eta -2\e^2 n i\p_s \eta+ic |\log\e| \e^2\p_s\eta)\zeta +2\e^2\p_s\eta\p_s \zeta.
\end{align*}
By using that, away from the vortices,
\begin{align*}
S_1(\phi)=S_1(iV_d\psi)=  iS_1(V_d)\psi +iV_d\bigl(\e^2\p^2_{ss}\psi+2\frac{\p_s V_d}{V_d}\p_s \psi-2in\p_s \psi+ic|\log \e| \e^2 \p_s \psi \bigr)
\end{align*}
we obtain
\begin{align}\label{eq:operator_S_1}
S_1(v)&= \left( \eta +(1-\eta)e^{i\psi} \right) \Bigl[S_1(V_d)+ S_1(\phi) \Bigr]  +(1-\eta)e^{i\psi}\bigl[ iV_di\e^2(\p_s\psi)^2-iS_1(V_d)\psi)\bigr] \nonumber \\
& \quad \quad+(\e^2 \p^2_{ss} \eta -2\e^2 n i\p_s \eta+ic |\log\e| \e^2\p_s \eta)\zeta +2\e^2\p_s\eta\p_s \zeta.
\end{align}
Putting together \eqref{eq:operator_S_1}  and \eqref{eq:S_equal_0_int1} we deduce that \(S(v)=0\) if and only if
\begin{equation}\begin{split}\label{eq:first_equivalence}
&\left( \eta +(1-\eta)e^{i\psi} \right) \Bigl[ S(V_d)+\Delta \phi+ S_1(\phi) -2\RE(\overline{V_d}\phi)V_d+(1-|V_d|^2)\phi\Bigr] 
-\eta \Bigl[ |\phi|^2(V_d+\phi)+2\RE(\overline{V_d}\phi)\phi \Bigr] \\
&\qquad +(1-\eta)iV_de^{i\psi} \Bigl[ i(\nabla \psi)^2+i\e^2(\p_s\psi)^2-|V_d|^2(e^{-2\IM(\psi)}-1+2\IM(\psi))-\frac{S(V_d)}{V_d}\psi\Bigr] \\
&\qquad+\eta (1-\eta) \left[\zeta G_1(\phi)+\overline{ \zeta}H_1(\phi)+\eta (\zeta G_2(\phi)+\overline{\zeta}H_2(\phi) \right]\\
&\qquad +(\Delta \eta +\e^2 \p^2_{ss} \eta -2\e^2 n i\p_s \eta+ic |\log\e| \e^2\p_s \eta)\zeta+2\e^2\p_s\eta\p_s \zeta+2 \nabla \eta \nabla \zeta=0.
\end{split}\end{equation}

\noindent We then divide the previous equation by \( \eta +(1-\eta)e^{i\psi}\). This term does not vanish if \( \|iV_d \psi\|_{L^\infty(\R^2)}\) is small enough. Indeed \( \eta +(1-\eta)e^{i\psi}=1+(1-\eta)(e^{i\psi}-1)\) and wherever \(\eta \neq 1\), \(V_d\) is a smooth function which does not vanish. Hence \(|\psi| \leq\frac{|iV_d\psi|}{|V_d|}\leq C \|\phi\|_{L^\infty(\R^2)}\) with \(\phi=iV_d\psi\). Thus \( (1-\eta)|e^{i\psi}-1| \leq C (1-\eta) |\psi| \leq C \|\phi\|_{L^\infty(\R^2)}\).

We observe that
\begin{align*}
\frac{(1-\eta)e^{i\psi}}{\eta +(1-\eta)e^{i\psi}}= (1-\eta) +\eta (1-\eta)\frac{e^{i\psi}-1}{\eta+(1-\eta)e^{i\psi}}.
\end{align*}
Thus, \eqref{eq:first_equivalence} becomes
\begin{multline}\label{eq:S_equl_0_int2}
E+L_0(\phi)-(1-\eta) iS(V_d)\psi -\frac{\eta}{\eta+(1-\eta)e^{i\psi}}\left(|\phi|^2(V_d+\phi)+2\RE(\overline{V_d} \phi)\phi \right) \\
+(1-\eta)iV_d\left[ i(\nabla \psi)^2+i\e^2(\p_s\psi)^2-|V_d|^2(e^{-2\IM(\psi)}-1+2\IM(\psi)) \right]+M_1(\phi)=0
\end{multline}
with \(E\) defined by \eqref{eq:def_error}, \(L_0\) defined in \eqref{L0} and
\begin{equation*}\begin{split}
M_1(\phi):=&\,\eta(1-\eta)\frac{e^{i\psi}-1}{\eta+(1-\eta)e^{i\psi}}\Bigl\{-iS(V_d) \psi
+iV_d \Bigl[i(\nabla \psi)^2+i\e^2(\p_s\psi)^2 \\
-& |V_d|^2(e^{-2\IM(\psi)}-1+2\IM(\psi)) \Bigr] \Bigr\} \\ &+\frac{\eta(1-\eta)}{\eta+(1-\eta)e^{i\psi}} \left[\zeta G_1(\phi)+\overline{ \zeta}H_1(\phi)+\eta (\zeta G_2(\phi)+\overline{\zeta}H_2(\phi) \right] \\
&+\frac{\Delta \eta +\e^2 \p^2_{ss} \eta -2\e^2 n i\p_s \eta+ic |\log\e| \e^2\p_s\eta}{\eta +(1-\eta)e^{i\psi}}\zeta+\frac{2\nabla \eta \nabla \zeta +2e^2\p_s \eta \p_s \zeta}{\eta +(1-\eta)e^{i\psi}}.
\end{split}\end{equation*}
 We note that \(M_1(\phi)\) is nonzero only when \(\eta (1-\eta) \neq 0\). Furthermore we can check that
 \begin{equation}\nonumber
 |M_1(\phi)| \leq C \|\psi\|^2_{C^1(B_1^2)} \leq C \|\phi\|^2_{C^1(B^2)}.
 \end{equation}
Now we use \eqref{eq:rel_L_and_L'} and we obtain that \(S(v)=0\) if and only if 
 \begin{equation}\begin{split}\label{eq:S_equal_0_int3}
 E+\eta L_0(\phi)+ &(1-\eta)iV_dL'(\psi)+(1-\eta)iV_d\left[ i(\nabla \psi)^2+i\e^2(\p_s\psi)^2-|V_d|^2(e^{-2\IM(\psi)}-1+2\IM(\psi))\right] \\
 &+\frac{\eta}{\eta+(1-\eta)e^{i\psi}} \left(|\phi|^2(V_d+\phi)+2\RE(\overline{V_d} \phi)\phi \right)+M_1(\phi)=0.
 \end{split}\end{equation}
Noticing that 
 \begin{equation}\nonumber
 \frac{\eta}{\eta +(1-\eta)e^{i\psi}}=1-\frac{(1-\eta)e^{i\psi}}{\eta +(1-\eta)e^{i\psi}} =\eta +\eta (1-\eta) \frac{1-e^{i\psi}}{\eta +(1-\eta)e^{i\psi}},
 \end{equation}  we write 
 \begin{multline}
  E+\eta L_0(\phi)+ (1-\eta)iV_dL'(\psi)+(1-\eta)iV_d\left[ i(\nabla \psi)^2+i\e^2(\p_s\psi)^2-|V_d|^2(e^{-2\IM(\psi)}-1+2\IM(\psi))\right] \\
 +\eta \left(|\phi|^2(V_d+\phi)+2\RE(\overline{V_d} \phi)\phi \right) +M_1(\phi)+M_2(\phi)=0,
 \end{multline}
 where
 \begin{align*}
M_2(\phi):=\eta (1-\eta) \frac{1-e^{i\psi}}{\eta +(1-\eta)e^{i\psi}} \left(|\phi|^2(V_d+\phi)+2\RE(\overline{V_d} \phi)\phi \right).
 \end{align*}
The same arguments used for \(M_1(\phi)\) show that \(M_2(\phi)\) is nonzero only when \(1 \leq \tilde{r}\leq 2\) and when \(\eta \neq 0\) and 
\begin{equation*}
|M_2(\phi)| \leq C \|\psi\|^2_{C^1(B_1^2)}\leq C \|\phi\|^2_{C^1(B^2)}.
\end{equation*}
Hence, by defining \(M(\phi):=M_1(\phi)+M_2(\phi)\) and
\begin{equation*}
N(\phi):= (1-\eta)iV_d\left[ i(\nabla \psi)^2+i\e^2(\p_s\psi)^2-|V_d|^2(e^{-2\IM(\psi)}-1+2\IM(\psi))\right]+M(\phi)
\end{equation*}
we obtain that \( S(v)=0\) if and only if \( E+\eta L_0(\phi)+(1-\eta)iV_dL'(\psi)-N(\phi)=0\) with \(N\) satisfying the desired properties.
\end{proof}

From the previous lemma, the problem we need to solve is 
\begin{equation}
\eta L_0(iV_d\psi)+(1-\eta)iV_dL'(\psi)=-E+\mathcal{N}(\psi) \text{ in } \times \R^+\times \R.
\end{equation}
With some abuse of notation we call 
\begin{equation}\label{def_mathcal_L_eps}
\L^\e(\phi):=\eta L_0(iV_d\psi)+(1-\eta)iV_dL'(\psi), \quad \psi=\frac{\phi}{iV_d}.
\end{equation}

\subsection{Another form of the equation near each vortex} 
In order to analyze the equation near each vortex, it will be useful to write it in a translated variable. Namely, we define 
$$\xi_j:=\begin{cases}
\xi_j^+&\quad \mbox{ for }j=1,\ldots,n^+,\\
\xi_{n^++1}:=0 & \quad\mbox{ if }n^-=1\;\mbox{ (i.e., the case of Theorem \ref{th:main2})}.
\end{cases}$$ 
We recall that \(d_\e\) is given by \eqref{eq:relationd}. Denote   $\z:=z-\xi_j$ and the function $\phi_j(\z)$ through the relation
\begin{equation}\label{defphi_j}
\phi_j(\z)=iw(\z)\psi(z), \ \ |\z|< d_\e.
\end{equation}
That is,
\begin{equation*}\nonumber
iV_d(z) \psi(z)=\phi_j(\z)\alpha_j(z),\qquad \mbox{where}\qquad \alpha_j(z):=\frac{V_d(z)}{w(z-\xi_j)}.
\end{equation*}
Hence in the translated variable the unknown \eqref{eq1} becomes, in $|\z|< d_\e$,
\begin{equation*}
v(z)=\alpha_j(z)\left(w(\z)+\phi_j(\z)+(1-\eta_1(\z))w(\z)\left[ e^{\frac{\phi_j(\z)}{w(\z)}}-1-\frac{\phi_j(\z)}{w(\z)}\right] \right).
\end{equation*}
We recall that from \eqref{eq:def_error} that \( E=S(V_d)\).
For $\phi_j, \psi$ linked through formula \eqref{defphi_j} we define
\begin{eqnarray}\label{defL_j}
L_j^\e(\phi_j)(\z)&:=&iw(\z)L'(\psi)(\z+\xi_j) =\frac{L_0(iV_d\psi)(z)}{\alpha_j(z)}-\frac{E(z)}{V_d(z)}\phi_j(\tilde{z}) \nonumber \\
&=&\frac{L_0(\phi_j(\z) \alpha_j(z))}{\alpha_j(z)}-\frac{E(z)}{V_d(z)}\phi_j(\tilde{z}),
\end{eqnarray}
with $L_0$ defined by \eqref{L0}.

Let us also define
\begin{equation*}\begin{split}
S_2(V):=&\, \p^2_{rr}V +\frac{1}{r}\p_rV+\frac{1}{r^2}\p_{ss}V+\e^2(\p^2_{ss}V -2ni\p_sV-n^2V)+ic \e^2|\log\e|(\p_sV -inV), \nonumber \\
S_3(V):=&\, \p^2_{rr}V +\frac{1}{r}\p_rV+\frac{1}{r^2}\p_{ss}V+\e^2(\p^2_{ss}V-2ni\p_sV)+ic \e^2|\log\e|\p_sV. \nonumber
\end{split}\end{equation*} Notice that
\begin{equation*}\nonumber
E(z)=S_2(\alpha_jw)+(1-|w|^2|\alpha_j|^2)w\alpha_j,
\end{equation*}
where we assume $\alpha_j$ and $w$ evaluated at $z$ and $\z$ respectively. Thus, using the equation satisfied by $w$,
\begin{equation*}\begin{split}
E=&\, wS_2(\alpha_j)+(1-|w|^2|\alpha_j|^2)\alpha_j w+2\nabla \alpha_j \nabla w+2\e^2\p_s\alpha_j \p_s w +\alpha_j S_3(w) \nonumber \\
=&\, wS_3(\alpha_j)-n^2\e^2w\alpha_j+c|\log\e|\e^2n\alpha_jw+(1-|w|^2|\alpha_j|^2)\alpha_j w+ 2\nabla \alpha_j \nabla w +2\e^2\p_s\alpha_j \p_s w\\
&\,+\alpha_j[\e^2(\p^2_{ss}w-2ni\p_sw)+ic \e^2|\log \e|\p_sw-(1-|w|^2)w ]. \nonumber
\end{split}\end{equation*}
This allows us to conclude
\begin{equation}\begin{split}\label{eq:defL_j}
L_j^\e(\phi_j) =& \, L_0(\phi_j)+\e^2(\p^2_{ss} \phi_j-2in\p_s \phi_j-n^2\phi_j)+ic|\log \e|\e^2(\p_s\phi_j-in\phi_j) \\
&  +2(1-|\alpha_j|^2)\RE(\overline{w}\phi_j)w -\Bigl(2\frac{\nabla \alpha_j}{\alpha_j}\frac{\nabla w}{w}+2\e^2\frac{\p_s \alpha_j}{\alpha_j}\frac{\p_s w}{w} \\
& \qquad +\e^2 \frac{(\p^2_{ss}w-2ni\p_sw)}{w}+ic\e^2|\log\e| \frac{\p_sw}{w}-n^2\e^2+n\e^2|\log\e| \Bigr)\phi_j \\
& +2\frac{\nabla \alpha_j}{\alpha_j}\nabla \phi_j+2\e^2\frac{\p_s \alpha_j}{\alpha_j}\p_s \phi_j,
\end{split}\end{equation}
Let us point out that, for $|\z|<d_\e$,
\begin{equation}\label{eq:estimatesonalpha_j}
 |\alpha_j(\z)|=1+O_\e(\e^2|\log \e|), \ \ \ \nabla \alpha_j(\z)=O_ \e(\e\sqrt{|\log \e|}), \ \ \  \Delta \a_j=O_\e(\e^2|\log \e|).
\end{equation}
With this in mind, we can see that the linear operator $L_j^\e$ is a small perturbation of $L_0$.

\subsection{Symmetry assumptions on the perturbation.}
Writing  $z=x_1+ix_2=re^{is}$ it can be seen that $V_d$  satisfies
\begin{equation*}
 V_d(x_1,-x_2)=\overline{V_d}(x_1,x_2) \quad \text{ and } V_d(e^{\frac{2i\pi}{n^+}}z)=V_d(z).
\end{equation*}
These symmetries are compatible with the solution operator $S$ defined in \eqref{S}: if $S(V)=0$ and $U(z):=\overline{V}(-x_1,x_2)$, then $S(U)=0$, and the same happens for $U(z):=\overline{V}(x_1,-x_2)$. Thus we look for a solution $V$ satisfying
\begin{equation*}
 V(x_1,-x_2)=\overline{V}(x_1,x_2), \quad V(e^{\frac{2i\pi}{n^+}}z)=V(z),
\end{equation*}
what is equivalent to ask
\begin{equation}\label{eq:eqsymmetriesofpsi}
\psi(x_1,-x_2)=-\overline{\psi}(x_1,x_2), \quad \psi(e^{\frac{2i\pi}{n^+}}z)=\psi(z).
\end{equation}

\section{Error estimates}\label{V}

The aim of this section is to compute the error of the approximation $V_d$ given by \eqref{eq:approx}. With this purpose, we divide the solution operator $S$ given in \eqref{SV} into three parts:
\begin{equation}\begin{split}\label{S}
S_a(V) & :=(\p^2_{rr}V+\frac{1}{r}\p_r V+\frac{1}{r^2}\p_{s s}^2 V)+(1-|V|^2)V, \\
S_b(V)& :=  \e^2(\p^2_{ss} V-2in\p_sV-n^2V), \\
S_c(V) &:= -ic|\log \e|\e^2(inV-\p_sV).  
\end{split}\end{equation}
Notice that $S_a$ corresponds to the solution operator for the Ginzburg-Landau equation in 2D. Likewise, $S_b$ represents the effect of the symmetry of the construction and $S_c$ the effect of working with a traveling wave in the Gross-Pitaevskii equation. 

By simplicity we denote
\begin{equation*}
w^i(z):=w(z-\xi^+_i), \quad w^j(z):=w(z-\xi^+_j),\quad 
w^k(z):=w(z-\xi^-_k),\quad w^l(z):=w(z-\xi^-_l),
\end{equation*}
i.e., we use the letters $i,j$ for the vortex with degree +1 and $k,l$ for the vortex of degree -1.

We will expand the error terms for the general case of $V_d$ given in \eqref{eq:approx}, so that they can be used for other constructions with a different number of filaments. Nevertheless, the estimates proved in the lemmas of this section correspond to the cases $n^-=0$ (see Theorem \ref{th:main1}) and $n^-=1$, $\xi_1^-=0$ (Theorem \ref{th:main2}).
\subsection{Size of the error \(S_a(V_d)\)}
Computing the gradient of the approximation:
\begin{align*}
\nabla V_d 
& = \sum_{i=1}^{n^+} \nabla w^i \prod_{j\neq i}w^j \prod_{k=1}^{n^-} \overline{w}^k+\sum_{k=1}^{n^-} \nabla \overline{w}^k \prod_{l\neq k} \overline{w}^l \prod_{i=1}^{n^+} w^i,
\end{align*}
we deduce
\begin{equation*}\begin{split}
\Delta V_d =&\left(\sum_{i=1}^{n^+} \Delta w^i \prod_{j\neq i} w^j+\sum_{i=1}^{n^+} \sum_{j\neq i} \nabla w^i \nabla w^j \prod_{m\neq i,j} w^m \right) \prod_{k=1}^{n^-} \overline{w}^k +2 \sum_{i=1}^{n^+} \sum_{k=1}^{n^-} \nabla w^i \nabla \overline{w}^k \prod_ {j\neq i} w^j \prod_{l\neq k} \overline{w}^k\\
&+\left(\sum_{k=1}^{n^-}\Delta \overline{w}^k \prod_{l\neq k} \overline{w}^l+\sum_{k=1}^{n^-} \sum_{l\neq k} \nabla \overline{w}^k \nabla \overline{w}^l \prod_{m\neq k,l} \overline{w}^m \right) \prod_{i=1}^{n^+}w^i .
\end{split}\end{equation*}

\noindent Thus we can write 
\begin{multline*}
\Delta V_d=V_d \,\Bigl(\sum_{i=1}^{n^+} \frac{\Delta w^i}{w^i}+\sum_{k=1}^{n^-} \frac{\Delta \overline{w}^k}{\overline{w}^k} 
+\sum_{i=1}^{n^+} \sum_{j\neq i} \frac{\nabla w^i}{w^i}\frac{\nabla w^j}{w^j} +\sum_{k=1}^{n^-} \sum_{l\neq k} \frac{\nabla \overline{w}^k}{\overline{w}^k}\frac{\nabla \overline{w}^l}{\overline{w}^l}+2\sum_{i=1}^{n^+} \sum_{k=1}^{n^-} \frac{\nabla w^i}{w^i}\frac{\nabla \overline{w}^k}{\overline{w}^k} \Bigr).
\end{multline*}
Recalling that $w$ solves \eqref{GL2} we find
\begin{align}\label{Sa_0}
S_a(V_d)&=V_d \,\Bigl\{ -\sum_{i=1}^{n^+} (1-|w^i|^2)-\sum_{k=1}^{n^-} (1-|w^k|^2)+\left( 1- \left| \prod_{i=1}^{n^+} \prod_{k=1}^{n^-} w^i \overline{w}^k\right|^2 \right) \\
&+\sum_{i=1}^{n^+} \sum_{j\neq i} \frac{\nabla w^i}{w^i}\frac{\nabla w^j}{w^j}+\sum_{k=1}^{n^-} \sum_{l\neq k} \frac{\nabla \overline{w}^k}{\overline{w}^k}\frac{\nabla \overline{w}^l}{\overline{w}^l}+2\sum_{i=1}^{n^+} \sum_{k=1}^{n^-} \frac{\nabla w^i}{w^i}\frac{\nabla \overline{w}^k}{\overline{w}^k} \Bigr\}.
\end{align}
We will use polar coordinates centered at \(\xi^+_j,\xi^-_k\), namely
\begin{equation}
z-\xi^+_j=r_je^{i\theta_j}, \quad z-\xi_k^-=r_ke^{i\theta_k}.
\end{equation}
Hence we can write
\begin{align*}
w^j & =w(z-\xi^+_j)=\r(r_j)e^{i\theta _j}=\r_j e^{i \theta_j}, \\
\overline{w}^k & =\overline{w}(z-\xi^-_k)=\r(r_k)e^{-i\theta_k}=\r_k e^{-i\theta_k}.
\end{align*}
We have
\begin{align}\label{derW}
w_{x_1}^j&= e^{i\theta_j}(\r'_j\cos \theta_j-i\frac{\r_j}{r_j}\sin \theta_j),\quad \quad w_{x_2}^j= e^{i\theta_j}(\r'_j \sin \theta_j+i\frac{\r_j}{r_j}\cos \theta_j), \\
\overline{w}_{x_1}^k&= e^{-i\theta_k}(\rho'_k\cos \theta_k+i\frac{\r_k}{r_k}\sin \theta_k),\quad \quad \overline{w}_{x_2}^k=e^{-i\theta_k}(\r'_k\sin \theta_k-i\frac{\r_k}{r_k}\cos \theta_k).
\end{align}
Hence 
\begin{equation*}\begin{split}
w_{x_1}^j w_{x_1}^l&=e^{i (\theta_j+\theta_l)} \Bigl\{ \r'_j \r'_l \cos \theta_j \cos \theta_l-\frac{\rho_j \rho_l}{r_j r_l}\sin \theta_j \sin \theta_l
-i \left( \frac{\rho'_j \rho_l}{r_l}\cos \theta_j \sin \theta_l +\frac{\rho'_l\rho_j}{r_j}\sin \theta_j \cos \theta_l \right) \Bigr\},\\
w_{x_2}^j w_{x_2}^l&=e^{i(\theta_j+\theta_l)}\Bigl\{ \r'_j\r'_l \sin \theta_j \sin \theta_l-\frac{\r_j \r_l}{r_j r_l}\cos \theta_j \cos \theta_l +i \left( \frac{ \r'_j \r_l}{r_l}\sin \theta_j \cos \theta_l+\frac{\r'_l \r_j}{r_j}\sin \theta_l \cos \theta_j \right) \Bigr\},
\end{split}\end{equation*}
and
\begin{equation*}\begin{split}
\frac{\nabla w^i}{w^i}\frac{\nabla w^j}{w^j}&=  \left(\frac{\r'_i\r'_j}{\r_i \r_j}-\frac{1}{r_i r_j}\right) \cos (\theta_i-\theta_j) +i\left(\frac{\r'_i}{\r_i r_j}-\frac{\r'_j}{\r_j r_i} \right)\sin(\theta_i-\theta_j),\\
\frac{\nabla \overline{w}^k}{\overline{w}^k}\frac{\nabla \overline{w}^l}{\overline{w}^l}&=\left( \frac{\r'_k\r'_l}{\r_k\r_l}-\frac{1}{r_kr_l}\right)\cos (\theta_k-\theta_l)-i\left(\frac{\r'_k}{\r_k r_k}-\frac{\r'_l}{\r_l r_k} \right)\sin(\theta_k-\theta_l).
\end{split}\end{equation*}
On the other hand,
\begin{equation*}\begin{split}
w_{x_1}^i \overline{w}_{x_1}^k&=e^{i(\theta_i-\theta_k)} \Bigl\{ \r'_i \r'_k \cos \theta_i \cos \theta_k+\frac{\r_i \r_k}{r_i r_k}\sin \theta_i \sin \theta_k +i\left( \frac{\r'_i\r_k}{r_k}\cos \theta_i \sin \theta_k-\frac{\r'_k\r_i}{r_i}\cos \theta_k \sin \theta_i\right) \Bigr\},\\
w_{x_2}^i \overline{w}_{x_2}^k&=e^{i(\theta_i-\theta_k)} \Bigl\{ \r'_i\r'_k \sin \theta_i \sin \theta_k+\frac{\rho_i \rho_k}{r_i r_k} \cos \theta_i \cos \theta_k -i \left( \frac{\r'_i \r_k}{r_k}\sin \theta_i \cos \theta_k-\frac{\r'_k \r_i}{r_i}\sin \theta_k \cos \theta_i\right) \Bigr\},
\end{split}\end{equation*}

\begin{equation*}
\frac{\nabla w^i}{w^i}\frac{\nabla \overline{w}^k}{\overline{w}^k}=\left(\frac{\r'_i\r'_k}{\r_i \r_k}+\frac{1}{r_i r_k} \right)\cos (\theta_i-\theta_k)+i \left( \frac{\r'_i}{\r_i r_k} +\frac{\r'_k}{\r_k r_i}\right)\sin (\theta_k-\theta_i).
\end{equation*}
Therefore, we can write the error as
\begin{equation}\begin{split}\label{eq:expression_error}
S_a(V_d) =&\,V_d\, \Bigl\{ -\sum_{i=1}^{n^+} (1-|w^i|^2)-\sum_{k=1}^{n^-}(1-|w^k|^2)+1-\left| \prod_{i=1}^{n^+} \prod_{k=1}^{n^-} w^i \overline{w}^k \right|^2\\
&+\sum_{i=1}^{n^+} \sum_{j\neq i} \left(\frac{\r'_i\r'_j}{\r_i\r_j}-\frac{1}{r_i r_j} \right)\cos (\theta_i-\theta_j)+i\left(\frac{\r'_i}{\r_i r_j}-\frac{\r'_j}{\r_j r_i} \right)\sin (\theta_i-\theta_j) \\
&+\sum_{k=1}^{n^-} \sum_{l\neq k} \left(\frac{\r'_k \r'_l}{\r_k \r_l}-\frac{1}{r_k r_l} \right)\cos ( \theta_k-\theta_l)-i\left(\frac{\r'_k}{\r_k r_l}-\frac{\r'_l}{\r_l r_k} \right)\sin (\theta_k-\theta_l) \\
&+2\sum_{i=1}^{n^+} \sum_{k=1}^{n^-} \left( \frac{\r'_i \r'_k}{\r_i \r_k}+\frac{1}{r_i r_k} \right)\cos (\theta_i-\theta_k)+i\left(\frac{\r'_i}{\r_i r_k}+\frac{\r'_k}{\r_k r_i} \right)\sin (\theta_k-\theta_i)\Bigr\}.
\end{split}\end{equation}

Let us define the total number of filaments as
$$N:=n^++n^-,$$ 
and by simplicity denote 
\begin{equation}\label{xij}
\begin{cases}
\xi_j=\xi_j^+ \quad \mbox{ for }j=1,\cdots,n^+,\\
\xi_N=0 \quad \mbox{ if }n^-=1,
\end{cases}
\end{equation}
and $r_je^{i\theta_j}:=re^{is}-\xi_j$, $j=1,\ldots, N$. Notice that $N$ will be $n^+$ and $n^++1$ for Theorem \ref{th:main1} and Theorem \ref{th:main2} respectively.

\begin{lemma}\label{errorA} Let us denote $E_a:=S_a(V_d)$, with $S_a$ and $V_d$ defined in \eqref{S} and \eqref{eq:approx} respectively.  There exists $C>0$ such that
\begin{equation}
\|E_a\|_{L^\infty(r_j<3)}\leq C\e \sqrt{|\log\e|}, \ \ \ \|\nabla E_a\|_{L^\infty(r_j<3)} \leq C\e \sqrt{|\log \e|}, \quad \text{ for all } j=1,\dots,N.
\end{equation}
Writing $E_a=iV_dR_a=iV_d(R_a^1+iR_a^2)$, in the region \( \bigcap_{j=1}^{N} \{ r_j >2\}\) we have 
\begin{eqnarray}
|R_a^1|\leq C\sum_{j=1}^{N}\frac{\e\sqrt{|\log \e|}}{r_j^3},\quad |\nabla R_a^1|\leq C\sum_{j=1}^{N}\frac{\e\sqrt{|\log\e|}}{r_j^4},\\
 |R_a^2|\leq C \sum_{j=1}^{N} \frac{\e\sqrt{|\log \e|}}{r_j}, \quad |\nabla R_a^2|\leq C \sum_{j=1}^{N} \frac{\e\sqrt{|\log \e|}}{r_j^2}.
\end{eqnarray}
Furthermore, 
 \[|R_a^2|\leq C \sum_{j=1}^{N}\frac{(\sqrt{|\log \e|}\e)^{\sigma}}{1+r_j^{2-\sigma}},\quad |\nabla R_a^2|\leq C \sum_{j=1}^{N}\frac{(\e\sqrt{|\log \e|})^\sigma}{1+r_j^{3-\sigma}},\quad \forall \,0<\sigma<1,\]
 in \( \bigcap_{j=1}^{N} \{ r_j >2\}\).
\end{lemma}

\begin{proof}
Suppose $n^-=0$. By symmetry it suffices to work in the angular sector
\begin{equation}\label{eq:angular_sector}
\Theta_1:=\left\{z\in \mathbb{C}:\, z=re^{is},\; r>0,\; s\in \left[-\frac{\pi}{n^+},\frac{\pi}{n^+}\right]\right\},
\end{equation}
where we have 
$$r_j \geq d_\e \sin \frac{\pi}{n^+}\geq C \e \sqrt{|\log \e|}\quad \mbox{ for all }j\neq 1.$$
To estimate the error near the vortex \(\xi_1\) we use expression \eqref{Sa_0} and the fact that from Lemma \ref{lem:propertiesofrho} we know
$$\rho_j=1-\frac{1}{2r_j^2}+O(r^{-4}_j)\quad\mbox{ and }\quad \rho'_j=\frac{1}{r_j^3}+O(r_j^{-4})\quad \mbox{ for }j \neq 1.$$
Far away from the vortex \(\xi_1\), i.e., in \( \{r_1>2\}\cap \Theta_1\) we use \eqref{eq:expression_error} and Lemma \ref{lem:propertiesofrho}  again. Note that to estimate the imaginary part \(R_a^2\) the dominant terms are of the form \(1/(r_1r_j)\) for \(r_j \neq 1\). Since \(r_j>r_1\) we can say that
\[ \frac{1}{r_1r_j}\leq \frac{1}{r_1^{2-\sigma}r_j^{\sigma}}\leq \frac{(\e\sqrt{|\log \e|)^\sigma}}{r_1^{2-\sigma}}.\]
The estimates for the gradient follow in the same way.
The case $n^-=1$ analogously follows by dividing the space into the regions closer to every vortex.
\end{proof}

\subsection{Size of the error \(S_b(V_d)\)}
We first note that
\begin{align}\label{estpsV}
\frac{\p_s V_d}{V_d}& =\sum_{j=1}^{n^+} \frac{\p_s w_j}{w_j}+\sum_{k=1}^{n^-}\frac{\p_s \overline{w}_k}{\overline{w}_k}, \\
\frac{\p^2_{ss}V_d}{V_d} &= \sum_{j=1}^{n^+} \frac{\p^2_{ss}w_j}{w_j}+\sum_{k=1}^{n^-} \frac{\p^2_{ss} \overline{w}_k}{\overline{w}_k}+\sum_{j=1}^{n^+}\sum_{l\neq j} \frac{\p_s w_j \p_s w_l}{w_j w_l}  +\sum_{k=1}^{n^-} \sum_{l\neq k} \frac{\p_s \overline{w}_k \p_s\overline{w}_l}{\overline{w}_k \overline{w}_l}+2\sum_{j=1}^{n^+}\sum_{k=1}^{n^-} \frac{\p_s w_j \p_s \overline{w}_k}{w_j \overline{w}_k},
\end{align}
and
\begin{align*}
\p_s w_j &=(\p_s r_j \r'_j+i\r_j \p_s \theta_j)e^{i\theta_j}, \\
\p^2_{ss} w_j &= \left( \p_{ss}^2r_j\r_j'+(\p_s r_j)^2\r_j''+2i\p_sr_j\r_j'\p_s\theta_j+i\r_j\p^2_{ss}\theta_j-\r_j(\p_s\theta_j)^2 \right)e^{i\theta_j}.
\end{align*}
Thus we find, after reorganizing the terms:
\begin{equation*}\begin{split}
\frac{S_b(V_d)}{\e^2V_d} =&\,\sum_{j=1}^{n^+} \p_{ss}^2 r_j \frac{\r'_j}{\r_j}+(\p_sr_j)^2\frac{\r_j''}{\r_j}+\sum_{k=1}^{n^-}\p^2_{ss}r_k\frac{\r'_k}{\r_k}+(\p_sr_k)^2\frac{\r_k''}{\r_k} 
-\left( \sum_{j=1}^{n^+}\p_s\theta_j-\sum_{k=1}^{n^-} \p_s\theta_k \right)^2\\
&+2n\left(\sum_{j=1}^{n^+}\p_s\theta_j-\sum_{k=1}^{n^-} \p_s\theta_k \right)-n^2 
+\sum_{j=1}^{n^+} \sum_{l=j} \p_s r_l \p_s r_j \frac{\r_j' \r_l'}{\r_j \r_l}+\sum_{k=1}^{n^-}\sum_{l\neq k} \p_s r_k \p_sr_l \frac{\r'_k \r'_l}{\r_k \r_l}\\
&+\sum_{j=1}^{n^+}\sum_{k=1}^{n^-} \p_sr_j\p_sr_k \frac{\r_j' \r'_k}{\r_j\r_k} +i \left\{ \sum_{j=1}^{n^+} \p^2_{ss} \theta_j-\sum_{k=1}^{n^-} \p^2_{ss} \theta_k +\sum_{j=1}^{n^+}2\p_sr_j\frac{\r'_j}{\r_j}\left(\sum_{l=1}^{n^+}\p_s\theta_l-\sum_{k=1}^{n^-} \p_s\theta_k-n \right)\right. \\
&\left.+ \sum_{k=1}^{n^-}2\p_s r_k \frac{\r'_k}{\r_k} \left( \sum_{l=1}^{n^-}\p_s\theta_l-\sum_{j=1}^{n^+} \p_s\theta_j-n \right) \right\}.
\end{split}\end{equation*}
We compute the derivative with respect to the variables \((r,s)\) of \(r_j,\theta_j,r_k,\theta_k\). Note that
\begin{equation}\label{eq:polar_relations}
r_je^{i\theta_j}=re^{is}-\xi_j^+=(re^{i(s-\varphi_j)}-|\xi_j^+|)e^{i\varphi_j},
\end{equation}
and hence 
\begin{eqnarray}
r_j^2 &=& r^2-2r|\xi_j^+|\cos(s-\varphi_j)+|\xi_j^+|^2 \nonumber \\
r_j\cos (\theta_j-\varphi_j) &=& r\cos(s-\varphi_j)-|\xi_j^+| \nonumber \\
r_j \sin (\theta_j -\varphi_j) &=& r\sin(s-\varphi_j). \nonumber
\end{eqnarray}
With the help of these relations we arrive at 
\begin{equation} \begin{split}\label{eq:derivative_rj}
&\p_sr_j =|\xi_j^+|\sin(\theta_j-\varphi_j),\qquad 
\p^2_{ss}r_j =|\xi_j^+|\cos(\theta_j-\varphi_j)+\frac{|\xi_j^+|^2}{r_j}\cos^2(\theta_j-\varphi_j), \\
&\p_s \theta_j = 1+\frac{|\xi_j^+|}{r_j}\cos(\theta_j-\varphi_j),\qquad
\p^2_{ss} \theta_j =-\frac{|\xi_j^+|}{r_j}\sin(\theta_j-\varphi_j)-\frac{2|\xi_j^+|^2}{r_j^2}\sin(\theta_j-\varphi_j)\cos(\theta_j-\varphi_j).
\end{split}\end{equation}
Analogous expressions hold for $\xi_k^-$. Using the fact \(n=n^+-n^-\) we observe that
\begin{multline}
-\left( \sum_{j=1}^{n^+}\p_s\theta_j-\sum_{k=1}^{n^-} \p_s\theta_k \right)^2+2n\left(\sum_{j=1}^{n^+}\p_s\theta_j-\sum_{k=1}^{n^-} \p_s\theta_k \right)-n^2 \\=-\left( \sum_{j=1}^{n^+} \frac{|\xi_j^+|}{r_j}\cos(\theta_j-\varphi_j)-\sum_{k=1}^{n^-}\frac{|\xi_k^-|}{r_k}\cos(\theta_k-\varphi_k) \right)^2,
\end{multline}
\begin{equation*}\begin{split}
\sum_{j=1}^{n^+} \p^2_{ss} \theta_j-\sum_{k=1}^{n^-} \p^2_{ss}\theta_k =&\left( \sum_{k=1}^{n^-} \frac{|\xi_k^-|}{r_k}\sin(\theta_k-\varphi_k)-\sum_{j=1}^{n^+}\frac{|\xi_j^+|}{r_j}\sin(\theta_j-\varphi_j)\right) \\
&-2\sum_{j=1}^{n^+} \frac{|\xi_j^+|^2}{r_j^2}\sin(\theta_j-\varphi_j)\cos(\theta_j-\varphi_j)+2\sum_{k=1}^{n^-} \frac{|\xi_k^-|^2}{r_k^2}\sin(\theta_k-\varphi_k)\cos(\theta_k-\varphi_k),
\end{split}\end{equation*}
\begin{equation}
\sum_{l=1}^{n^+}\p_s\theta_l-\sum_{k=1}^{n^-} \p_s\theta_k-n=\sum_{l=1}^{n^+}\frac{|\xi_l^+|}{r_l}\cos(\theta_l-\varphi_l)-\sum_{k=1}^{n^-} \frac{|\xi_k^-|}{r_k}\cos(\theta_k-\varphi_k).
\end{equation}
We can now estimate the size of this part of the error, by taking 
$$|\xi_j^+|=d_\e=\frac{\d_\e}{\e|\log \e|^{1/2}}\quad \forall  \,1\leq j \leq n^+,\quad  \varphi_j=2i\pi (j-1)/n,\quad |\xi_k^{-}|=0,\quad\varphi_k=0.$$
We find
\begin{equation}\begin{split}\label{eq:error_S_b}
\frac{S_b(V_d)}{\e^2V_d}=& \frac{\d_\e}{\e |\log \e|^{1/2}}\sum_{j=1}^{n^+} \left(\frac{\r_j'}{\r_j}\cos(\theta_j-\varphi_j)-i\frac{\sin(\theta_j-\varphi_j)}{r_j} \right) \\ 
&+\frac{\d_\e^2}{\e^2|\log \e|} \Bigl[\sum_{j=1}^{n^+} \left( \frac{\r_j''}{\r_j}\sin^2(\theta_j-\varphi_j)+\left(\frac{\r'_j}{r_j\r_j}-\frac{1}{r_j^2} \right)\cos^2(\theta_j-\varphi_j) \right) \\
& \quad -2i \frac{1}{r_j^2} \sin(\theta_j-\varphi_j)\cos(\theta_j-\varphi_j)\Bigr] \\
&+\frac{\d_\e^2}{\e^2|\log \e|}\left[-2\sum_{j=1}^{n^+}\sum_{l\neq j} \frac{\cos(\theta_j-\varphi_j)\cos(\theta_l-\varphi_l)}{r_jr_l}+\sum_{j=1}^{n^+}\sum_{l\neq j} \frac{\sin(\theta_j-\varphi_j) \sin(\theta_l-\varphi_l)\r'_j\r'_l}{\r_j\r_l}\right. \\ 
&\qquad\qquad  \left.+2i\sum_{j=1}^{n^+}\sum_{l=1}^{n^+} \frac{\sin(\theta_j-\varphi_j)\cos(\theta_l-\varphi_l)\rho_j'}{{\rho_j}r_l} \right].
\end{split}\end{equation}

\begin{lemma}\label{errorB}
Let us denote $E_b:=S_b(V_d)$, with $S_b$ and $V_d$ defined in \eqref{S} and \eqref{eq:approx} respectively. There exists $C>0$ such that
\begin{equation*}
\|E_b\|_{L^\infty(r_j<3)}\leq \frac{C}{|\log \e|}, \qquad \|\nabla E_b\|_{L^\infty(r_j<3)} \leq \frac{C}{|\log \e|},\quad \mbox{ for }j=1,\ldots,N.
\end{equation*}
Furthermore, writing $E_b=iV_dR_b=iV_d(R_b^1+iR_b^2)$, in the region \( \cap_{j=1}^{N} \{ r_j >2\}\) we have
\begin{eqnarray*}
|R_b^1| \leq \frac{C}{|\log \e|} \sum_{j=1}^{N} \frac{1}{r_j^2},  \ \ |\nabla R_b^1|  \leq \frac{C}{|\log \e|}\sum_{j=1}^{N} \frac{1}{r_j^3},\\ 
|R_b^2|\leq \frac{C}{|\log \e|}\sum_{j=1}^{N} \frac{1}{r_j^2},  \ \ |\nabla R_b^2|  \leq \frac{C}{|\log \e|}\sum_{j=1}^{N} \frac{1}{r_j^3}.
\end{eqnarray*}

\end{lemma}
\begin{proof}
Applying the properties of $\rho_j$ stated in Lemma \ref{lem:propertiesofrho} at the identity \eqref{eq:error_S_b} it easily follows that  
\begin{equation*}
\|E_b\|_{L^\infty(r_j<3)}\leq \frac{C}{|\log \e|}, \qquad \|\nabla E_b\|_{L^\infty(r_j<3)} \leq \frac{C}{|\log \e|}.
\end{equation*}
Thanks to Lemma \ref{lem:propertiesofrho}, to prove the estimates far from the vortices the only difficult term is 
\begin{equation*}
A:=\frac{\e \hat{d}_\e}{|\log\e|^{1/2}}\sum_{j=1}^{n^+}\frac{\sin(\theta_j-\varphi_j)}{r_j}.
\end{equation*}
Notice that in the region $3<r_j<2/(\e|\log\e|^{1/2})$ we directly obtain 
$$|A|\leq \frac{C}{|\log\e|}\frac{1}{1+r_j^2}.$$
Finally, in the case $r_j>2/(\e|\log\e|^{1/2})$ we use the fact \(r\sin(s-\varphi_j)=r_j\sin(\theta_j-\varphi_j)\) to write
$$A=\frac{\e \hat{d}_\e}{|\log\e|^{1/2}}\sum_{j=1}^{n^+}\frac{r}{r_j^2}\sin(s-\varphi_j),$$
and the result follows by expanding
$$\frac{r}{r_j^2}=\frac{1}{r}-\frac{2}{r^2}\frac{\hat{d}_\e}{\e|\log\e|^{1/2}}\cos(s-\varphi_j)+O\left(\frac{d_\e^2}{r^3}\right),$$
and noticing that
\[ \sum_{j=1}^{n^+} \sin(s-\varphi_j)=0,\]
since \(\xi_j^+\) are the \(j\)-th root of the unity.
\end{proof}

\subsection{Size of the error \(S_c(V_d)\)}
Using the equality in \eqref{estpsV} we deduce
\begin{equation}
\frac{S_c(V_d)}{iV_dc \e^2|\log \e|}= \sum_{j=1}^{n^+}\frac{\p_s r_j \rho'_j}{\rho_j}+\sum_{k=1}^{n^-} \frac{\p_s r_k \rho_k'}{\rho_k}+i\left( \sum_{j=1}^{n^+}\p_s\theta_j -\sum_{k=1}^{n^-} \p_s\theta_k- n\right),
\end{equation}
and applying \eqref{eq:derivative_rj} we get
\begin{equation}\begin{split}\label{eq:errorSc}
\frac{S_c(V_d)}{iV_dc \e^2|\log \e|}=&\sum_{j=1}^{n^+} |\xi_j^+|\sin(\theta_j-\varphi_j)\frac{\rho_j'}{\rho_j}+\sum_{k=1}^{n^-} |\xi_k^-|\sin(\theta_k-\varphi_k)\frac{\rho_k'}{\rho_k}\\
&+i\left( \sum_{j=1}^{n^+} \frac{|\xi_j^+|\cos(\theta_j-\varphi_j)}{r_j}-\sum_{k=1}^{n^-} \frac{|\xi_k^-|\cos(\theta_k-\varphi_k)}{r_k} \right).
\end{split}\end{equation}
\begin{lemma}\label{errorC}
Let us denote $E_c:=S_c(V_d)$, with $S_c$ and $V_d$ defined in \eqref{S} and \eqref{eq:approx} respectively. There exists $C>0$ such that
\begin{equation*}
\|E_c\|_{L^\infty(r_j<3)}\leq C\e \sqrt{|\log\e|}, \qquad \|\nabla E_c\|_{L^\infty(r_j<3)} \leq C\e \sqrt{|\log \e|},\quad \mbox{ for every }j=1,\ldots, N.
\end{equation*}
Furthermore, writing $E_c=iV_dR_c=iV_d(R_c^1+iR_c^2)$, in the region \(\cap_{j=1}^{N} \{r_j>2\}\) we have
\begin{eqnarray}\label{estSinAlpha}
|R_c^1|\leq C \e \sqrt{|\log\e|} \sum_{j=1}^{N}\frac{1}{r_j^3}, \quad |\nabla R_c^1|\leq C \e \sqrt{|\log\e|} \sum_{j=1}^{N}\frac{1}{r_j^4}, \\
 |R_c^2|\leq C \e\sqrt{|\log\e|} \sum_{j=1}^{N}\frac{1}{r_j}, \quad |\nabla R_c^2|\leq C \e\sqrt{|\log\e|} \sum_{j=1}^{N}\frac{1}{r_j^2}.
\end{eqnarray}
We also have
\begin{equation}\label{estConAlpha}
|R_c^2|\leq C(\e\sqrt{|\log \e|})^\sigma \sum_{j=1}^{N}\frac{1}{r_j^{2-\sigma}},\quad |\nabla R_c^2|\leq C(\e\sqrt{|\log \e|})^\sigma\sum_{j=1}^{N}\frac{1}{r_j^{3-\sigma}}
\end{equation}
for every \(0<\sigma<1\).
\end{lemma}
\begin{proof}
The estimates for $E_c$ near the vortices and \eqref{estSinAlpha} follow straightforward from \eqref{eq:errorSc} and Lemma \ref{lem:propertiesofrho}. To see \eqref{estConAlpha} we divide the analysis into two regions. Assume $r_j=\min\{r_l:\,l=1,\ldots,n^+\}$. If $2<r_j\leq 2/(\e|\log\e|^{1/2})$, from \eqref{eq:errorSc} we obtain
$$|R_c^2|\leq C\frac{(\e|\log\e|^{1/2})^\sigma}{r_j^{2-\sigma}},\quad \mbox{ for every }0<\sigma<1.$$
If $r_j \geq  2/(\e|\log\e|^{1/2})$, using that $r\cos(s-\varphi_j)=r_j\cos(\theta_j-\varphi_j)$ and proceeding as in the proof of Lemma \ref{errorB} we conclude
$$|R_c^2|\leq \frac{C}{r^2}\leq C\frac{(\e|\log\e|^{1/2})^\sigma}{r_j^{2-\sigma}},\quad \mbox{ for every }0<\sigma<1.$$
The estimate for the gradient follows analogously.
\end{proof}

We define 
\begin{align}
\label{Reps}
R_\varepsilon := \frac{\alpha_0}{\varepsilon |\log \varepsilon|},
\end{align}
with \(\alpha_0>0\) a constant to be determined later. Note that \(|\log R_\e| =|\log \e|(1+o_\e(1))\) and \(R_\e\ll d_\e\). We also define 
 the norm
\begin{align}
\|h\|_{**}
&:=\sum_{j=1}^N \| V_d h\|_{C^\alpha(r_j<3)} \nonumber
\\
& \qquad + \sup_{r_j>2,\,1\leq j\leq N} 
\left[| \RE(h)|\left(\sum_{j=1}^N r_j^{-2}+\e^2\right)^{-1}+|\IM(h)|\left(\sum_{j=1}^N r_j^{-2+\sigma}+\varepsilon^{\sigma-2}\right)^{-1}\right]\nonumber
\\
& \qquad 
+ \sup_{2<|z-\xi_j|<2R_\varepsilon , \, 1\leq j\leq N}[ \RE(h)]_{\alpha,B_{|z|/2}(z)}\left(\sum_{j=1}^N |z-\xi_j|^{-2-\alpha}\right)^{-1} \nonumber
\\
& \qquad 
+ \sup_{2<|z-\xi_j|<2R_\varepsilon , \, 1\leq j\leq N}[ \IM(h)]_{\alpha,B_{1}(z)}\left(\sum_{j=1}^N |z-\xi_j|^{-2+\sigma}\right)^{-1}, \label{def:norm_**}
\end{align}
where \(r_j=|z-\xi_j|\), $\|f\|_{C^\alpha(D)} = \|f\|_{C^{0,\alpha}(D)}$, and 
\begin{align}
[f]_{\alpha,D}
& := \sup_{x,y\in D , \, x\not=y}\frac{|f(x)-f(y)|}{|x-y|^\alpha}, \label{def:Holder_norm}
\\
\|f\|_{C^{k,\alpha}(D)} 
&:=
\sum_{j=0}^k \|D^jf\|_{L^\infty	(D)} \label{def:Ckalpha_norm}
+ [D^k f]_{\alpha,D} .
\end{align}
This norm will be the appropriate setting in the right hand side of the problem in order to prove the invertibility result stated in Proposition \ref{prop:linearfull}. Its precise form is determined by the decay of the error terms, as we identified in Lemmas \ref{errorA}, \ref{errorB} and \ref{errorC}.
Putting these together we can summarize the size of the error of the approximation measured in this norm in the following result.
\begin{proposition}\label{globalError}
Consider $R$ defined as \(S(V_d)=iV_dR\), with $S$ and $V_d$ given by  \eqref{S} and  \eqref{eq:approx} respectively. Then,
\begin{equation*}
\|R\|_{**} \leq \frac{C}{|\log \e|}.
\end{equation*}
\end{proposition}

\subsection{Decomposition of the error}

Recall the notation in polar coordinates $z-\xi_j = r_j e^{i\theta_j}$, with $\xi_j$ defined in \eqref{xij}.
We  can decompose a function \(h\) satisfying \(h(\overline{z})=-\overline{h}(z)\) in Fourier series in $\theta_j$ as
\begin{align}
& h = \sum_{k=0}^\infty h^{k,j}, \label{def:h_Fourier}\\
& h^{k,j}(r_j,\theta_j) = h_1^{k,j}(r_j) \sin (k\theta_j) + i h_2^{k,j}(r_j) \cos(k\theta_j), \quad
h_1^{k,j}(r_j), h_2^{k,j}(r_j) \in \R , \nonumber
\end{align}
and define
\begin{align*}
h^{e,j} = \sum_{k \text{ even}} h^{k,j}, \quad
h^{o,j} = \sum_{k \text{ odd}} h^{k,j}.
\end{align*}
Let $\mathcal{R}_j$ denote the reflection across the line $\RE(z) =\RE(\xi_j)$.  We have
\begin{equation}\label{def:reflection_R_j}
\mathcal{R}_j z = 2\RE(\xi_j)-\RE(z) + i \IM(z).
\end{equation}

\noindent Then $h^{e,j}$ and $h^{o,j}$ have the symmetries
\begin{align}
\nonumber
h^{o,j} (\mathcal{R}_j z ) =  \overline{h^{o,j}(z)} ,\quad \quad  h^{e,j}(\mathcal{R}_j z ) = -\overline{h^{e,j}(z)},
\end{align}
and we can define equivalently
\begin{align}
\label{hoj}
h^{o,j}(z) &= \frac{1}{2}[ h(z) + \overline{h(\mathcal{R}_j z )}] , \quad
h^{e,j}(z) = \frac{1}{2}[ h(z) - \overline{h(\mathcal{R}_j z )}] .
\end{align}
It is convenient to consider a global function  $h^o$ defined as follows: we introduce cut-off functions $\eta_{j,R}$, as
\begin{align}
\label{def:etajR}
\eta_{j,R}(z) := \eta_1 \Bigl(\frac{|z-\xi_j|}{R}\Bigr), 
\end{align}
where $\eta_1:\R\to[0,1]$ is a smooth function such that $\eta_1(t) = 1$ for $t \leq 1$ and $\eta_1(t)=0$ for $t\geq 2$. Consider $R_\e$ given in \eqref{Reps}, and $\alpha_0>0$ a small fixed constant so that $R_\varepsilon \leq \frac{1}{2}d_\e$. For any $h:\mathbb{C}\to \mathbb{C}$ we define
\begin{align}
\label{def-ho}
h^o &:= \sum_{j=1}^N \eta_{j,R_\varepsilon} h^{o,j},\qquad
h^e :=h-h^o. 
\end{align}
We introduce the new semi-norm
\begin{align}
|h|_{\sharp\sharp}
:= \sum_{j=1}^N \|  V_d h \|_{C^{0,\alpha}(r_j<4)}+\sup_{2<r_j<R_\e, \,1\leq j\leq N}\left[|\RE(h)|\left(\sum_{j=1}^N r_j^{-1}\right)^{-1}+|\IM(h)|\left(\sum_{j=1}^N r_j^{-1+\sigma}\right)^{-1}\right], \label{def:semi_norm_sharpsharp}
\end{align}
with $0<\alpha,\sigma<1$ constant to be chosen later. This semi-norm is devoted to identify some elements of the error with less decay but smaller size than the general term measured in the norm $\|\cdot\|$. This observation will allow us to obtained a more refined a priori estimate (see Proposition \ref{prop:sharp2b}), which will be a key point in the fixed point argument performed in Proposition \ref{prop:nonlinear}.

\begin{lemma}\label{lemma:decomposition}
Let $V_d$ given by \eqref{eq:approx} and denote $S(V_d)=E=iV_d R.$ Then we can write
$$R=R^o+R^e, \quad R^o=\hat{R}^o+\tilde{R}^o,$$
with \(R^o\) defined as in \eqref{def-ho} and
\(R^o(\mathcal{R}_jz) =\overline{R^o (z)} \) in \(\bigcup_{j=1}^{n^+}B_{R_\e}(\xi_j^+)\),
\begin{align*}
|\hat{R}^o|_{\sharp\sharp} & \leq C \frac{\e}{\sqrt{|\log \e|}}, \quad \|\tilde{R}^o\|_{**} \leq C\e\sqrt{|\log \e|}, \quad \|R^e\|_{**}+\|R^o\|_{**}\leq \frac{C}{|\log \e|}
\end{align*}
\end{lemma}

\begin{proof}
From Proposition \ref{globalError} we immediately obtain 
$$\|R^e\|_{**}+\|R^o\|_{**}\leq \frac{C}{|\log \e|}.$$
For every $j\in\{1,\ldots, n^+\}$ we define
\begin{equation*}\begin{split}
R_j:=&\,\frac{-\d_\e\e}{|\log\e|^{1/2}}\sum_{l=1}^{n^+}\frac{\sin(\theta_l-\varphi_l)}{r_l}-\frac{2\d_\e^2}{|\log\e|}\sum_{l\neq j}\frac{\sin(\theta_l-\varphi_l)\cos(\theta_l-\varphi_l)}{r_l^2}\\
&+i\frac{\d_\e^2}{|\log\e|}\left\{\sum_{l\neq j}\frac{\cos^2(\theta_l-\varphi_l)}{r_l^2}+2\sum_{l=1}^{n^+}\sum_{k\neq j}\frac{\cos(\theta_l-\varphi_l)\cos(\theta_k-\varphi_k)}{r_lr_k}\right\},
\end{split}\end{equation*}
and, according to \eqref{hoj} and \eqref{def-ho},
$$R^{o,j}:=\frac{1}{2}[ R_j(z) + \overline{R_j(\mathcal{R}_j z )}],\qquad \hat{R}^o:=\sum_{j=1}^{n^+} \eta_{j,R_\varepsilon} R^{o,j},$$
with $\eta_{j,R_\varepsilon}$ given in \eqref{def:etajR} and $R_\e$ in \eqref{Reps}. 

We can check that \( |R^{o,j}|_{\sharp \sharp} \leq C\e \sqrt{|\log \e|}\). This is because when looking at the vortex \(j\), \(r_l=O\left(1/(\e \sqrt{|\log \e|})\right)\) for \(l \neq j\). Thus
\begin{align*}
|\hat{R}^o|_{\sharp \sharp} &\leq \max_{1\leq j\leq n^+} \|\eta_{j,R_\e}\|_{L^\infty} \sum_{j=1}^{n^+} |R^{o,j}|_{\sharp \sharp} \leq \sum_{j=1}^{n^+} |R^{o,j}|_{\sharp \sharp} \leq \frac{C\e}{\sqrt{|\log \e|}}.
\end{align*}
Now we define $\tilde{R}^o:=R^o-\hat{R}^o$. We recall that \( R^o=\left(\frac{S_a(V_d)+S_b(V_d)+S_c(V_d)}{iV_d}\right)^o\). It follows from Lemma \ref{errorA} and  Lemma \ref{errorC} that
\begin{equation*}
\left\| \left(\frac{S_a(V_d)+S_b(V_d)}{iV_d}\right)^o \right\|_{**} \leq C \e \sqrt{|\log \e|}.
\end{equation*}
On the other hand, using \eqref{eq:error_S_b}, we find that
\begin{multline*}
\left(\frac{S_b(V_d)}{iV_d} \right)^o-\hat{R}^o=\sum_{j=1}^{n^+} \eta_{j,R_\e}\Bigl\{ \frac{2\hat{d}_\e}{|\log \e|} \sum_{k=1}^{n^+} \sum_{l=1}^{n^+} \frac{\sin(\theta_k-\varphi_k) \cos(\theta_l-\varphi_l)\rho '_k}{\rho_k r_l}-i\Bigl[\frac{\e \hat{d}_\e}{|\log \e|^{1/2}} \sum_{k=1}^{n^+} \frac{\rho'_k}{\rho_k}\cos(\theta_k-\varphi_k)  \\
+\frac{\hat{d}_\e^2}{|\log \e|}\left( \sum_{k=1}^{n^+} \frac{\rho_k''}{\rho_k}\sin^2(\theta_k- \varphi_k)+\frac{\rho'_k}{r_k \rho_k}\cos^2(\theta_k-\varphi_k) +\sum_{k=1}^{n^+} \sum_{l \neq k} \frac{\sin(\theta_k-\varphi_k)\sin(\theta_l-\varphi_l)\rho'_k \rho'_l}{\rho_k\rho_l} \right) \Bigr] \Bigl\}^{o,j}.
\end{multline*}
From the last expression we can obtain that \[ \left\|\left(\frac{S_b(V_d)}{iV_d} \right)^o-\hat{R}^o \right\|_{**} \leq C\e \sqrt{|\log \e|}.\] In order to see this we use that, for \(j\) fixed, expressions of the type \( \frac{\rho_j''}{\rho_j}\sin^2(\theta_j-\varphi_j)+\frac{\rho'_j}{r_j \rho_j} \cos^2(\theta_j-\varphi_j)\) do not appear in the odd decomposition whereas for   \(k\neq j\) a similar expression with \(j\) replaced by \(k\) has the desired size since \(r_k\geq C (\e \sqrt{|\log \e|})^{-1}\). This concludes the proof.
\end{proof}

\begin{lemma}\label{decError}
Let $V_d$ be the approximation given by \eqref{eq:approx}, and $S_b$ and $S_c$ the operators defined at \eqref{S}. Then, for every \(1\leq j\leq n^+\) we have
\begin{equation*}\label{eq:prep_reduction}\begin{split}
\frac{S_b(V_d)}{V_d}=&\,\frac{\d_\e}{|\log\e|}\frac{w_{x_2 x_2}^j(r_je^{i(\theta_j-\varphi_j)})}{w^j(r_je^{i(\theta_j-\varphi_j)})}+\frac{\d_\e \e }{\sqrt{|\log \e|}} \frac{w_{x_1}^j(r_je^{i(\theta_j-\varphi_j)})}{w^j(r_je^{i(\theta_j-\varphi_j)})} 
+G^j_b,\\
\frac{S_c(V_d)}{V_d}=&\,ic\d_\e \e\sqrt{|\log\e|}\frac{w_{x_2}^j(r_je^{i(\theta_j-\varphi_j)})}{w^j(r_je^{i(\theta_j-\varphi_j)})}+G^j_c,
\end{split}
\end{equation*}
with 
\begin{equation*}\begin{split}\label{eq:rest_in_recduction1}
\RE\int_{B(\xi_j^+,R_\e)}w_{x_2 x_2}^j(r_je^{i(\theta_j-\varphi_j)})\overline{w}^j_{x_1}(r_je^{i(\theta_j-\varphi_j)})&=0,
\end{split}\end{equation*}
and
\begin{equation*}\begin{split}\label{eq:rest_in_recduction2}
\RE \int_{B(\xi_j^+,R_\e)} w^j(r_je^{i(\theta_j-\varphi_j)})G_b^j\overline{w}^j_{x_1}(r_je^{i(\theta_j-\varphi_j)})&=O\left(\frac{\e}{\sqrt{|\log \e|}} \right),\\
\RE \int_{B(\xi_j^+,R_\e)} w^j(r_je^{i(\theta_j-\varphi_j)})G_c^j\overline{w}^j_{x_1}(r_je^{i(\theta_j-\varphi_j)})&=O\left(\frac{\e}{\sqrt{|\log \e|}} \right).
\end{split}\end{equation*}
\end{lemma}
\begin{proof}
Let us fix $j$, $1\leq j\leq n^+$. Using Lemma \ref{lem:propertiesofrho}, from expression \eqref{eq:error_S_b} we can identify the principal terms and write
\begin{equation*}\begin{split}
\frac{S_b(V_d)}{V_d}=&\,\frac{\d_\e}{|\log\e|}\left[\frac{\r''_j}{\r_j}\sin^2(\theta_j-\varphi_j)+\left(\frac{\r'_j}{r_j\r_j}-\frac{1}{r_j^2}\right)\cos^2(\theta_j-\varphi_j)\right.\\
&\qquad\left.+2i\left(\frac{\r'_j}{r_j\r_j}-\frac{1}{r_j^2}\right)\sin(\theta_j-\varphi_j)\cos(\theta_j-\varphi_j)\right]\\
&+\frac{\d_\e\e}{|\log\e|^{1/2}}\left[\frac{\r'_j}{r_j}\cos(\theta_j-\varphi_j)-i\frac{\sin(\theta_j-\varphi_j)}{r_j}\right]+G_b^j,
\end{split}\end{equation*}
that this is the same as 
$$\frac{S_b(V_d)}{V_d}=\frac{\d_\e}{|\log\e|}\frac{w_{x_2 x_2}^j(r_je^{i(\theta_j-\varphi_j)})}{w^j(r_je^{i(\theta_j-\varphi_j)})}+\frac{\d_\e \e }{\sqrt{|\log \e|}} \frac{w_{x_1}^j(r_je^{i(\theta_j-\varphi_j)})}{w^j(r_je^{i(\theta_j-\varphi_j)})}+G_b^j.$$
Likewise, from \eqref{eq:errorSc} we can divide
\begin{equation*}\begin{split}
\frac{S_c(V_d)}{V_d}=&\,ic\d_\e \e\sqrt{|\log\e|}\left(\frac{\rho_j'}{\r_j}\sin(\theta_j-\varphi_j)+i\frac{1}{r_j}\cos(\theta_j-\varphi_j)\right)+ G_c^j\\
=&\,ic\d_\e \e\sqrt{|\log\e|}\frac{w_{x_2}^j(r_je^{i(\theta_j-\varphi_j)})}{w^j(r_je^{i(\theta_j-\varphi_j)})}+G_c^j.
\end{split}\end{equation*}
Finally, \eqref{eq:rest_in_recduction1} holds using the formulae for \(w_{x_1}^j, w^j_{x_2x_2}\), the evenness of the cosinus and the oddness of sinus.
\end{proof}

\section{A projected problem}\label{VI}
For the sake of simplicity, in this section we will restrict ourselves to the case of Theorem \ref{th:main1}, that is, $n^+=n$ and $n^-=0$. The case of Theorem \ref{th:main2} follows with straightforward adaptations.

The final goal of this section is to prove existence of solution of the projected problem 
\begin{equation}\label{eq:nonlinear}
\left\{
\begin{split}
&\L^\e(\phi)=-E+N(\phi)+iV_d\sum_{j=1}^{n^+} \left\{c^1_{j}\frac{\chi_j w_{x_1}(z-\xi^+_j)}{iw(z-\xi^+_j)} +c^2_{j}\frac{\chi_j w_{x_2}(z-\xi^+_j)}{iw(z-\xi^+_j)} \right\}\;\; \text{ in } \R^2, \\
&\RE \int_{\R^2} \chi \overline{\phi_j}w_{x_1}=\RE \int_{\R^2} \chi \overline{\phi_j}w_{x_2}=0, \ \text{ with }\phi_j(z)=iw(z)\frac{\phi(z+\xi_j^+)}{iV_d(z+\xi_j^+)}, \ \ j=1,\cdots, n^+, \\
&\phi \text{ satisfies } \eqref{eq:eqsymmetriesofpsi},
\end{split}
\right.
\end{equation}
where \(\L^\e\) is defined in \eqref{def_mathcal_L_eps} and \(N \) is defined in \eqref{def:N} and
\begin{equation}
\nonumber
\chi(z):=\eta_1\left(\frac{|z|}{2} \right), \ \ \ \chi_j(z):=\eta_1\left(\frac{|z-\xi^+_j|}{2} \right),
\end{equation}
with $\eta_1$ a smooth cut-off function such that $\eta_1(t)=1$ if $t\leq 1$ and $\eta_1(t)=0$ if $t\geq2$.

To do so, we will start by considering a linear projected version. Indeed, given a function $h$  satisfying the symmetries \eqref{eq:eqsymmetriesofpsi} and with an appropriate decay, our first aim  is to solve the linear equation
\begin{align}
\label{eq:linear}
\left\{
\begin{aligned}
& \L^\e(\phi)=iV_dh+iV_d\sum_{j=1}^{n^+} \left\{c^1_{j}\frac{\chi_j w_{x_1}(z-\xi^+_j)}{iw(z-\xi^+_j)} +c^2_{j}\frac{\chi_j w_{x_2}(z-\xi^+_j)}{iw(z-\xi^+_j)} \right\}
\quad\text{in }\R^2,
\\
& \RE\int_{B(0,4)} \chi \overline{\phi_j}w_{x_1}=\RE\int_{B(0,4)} \chi \overline{\phi_j}w_{x_2}=0, \text{ with }\phi_j(z)=iw(z)\frac{\phi(z+\xi_j^+)}{iV_d(z+\xi_j^+)}, \ \ j=1, \cdots,n^+,\\
&\phi \text{ satisfies the symmetry }  \eqref{eq:eqsymmetriesofpsi},
\end{aligned}
\right.
\end{align}
We remark that the elements \( w_{x_1},w_{x_2}, iw\) are the basis of the kernel of the linearized Ginzburg-Landau operator around the standard vortex \(w\) in a natural energy space, cf.\ \cite{DFK}. A priori we should add also the projections on the elements \(i\chi_j w(z-\xi^+_j)\) and ask an orthogonality condition with respect to \(i\chi_jw(z-\xi_j^+)\). However, thanks to the symmetry assumptions \eqref{eq:eqsymmetriesofpsi}, the orthogonality condition with respect to \(iw(z-\xi_j^+)\) is automatically satisfied.  Furthermore, also thanks to these symmetry assumptions and to the symmetry of the operator \(\mathcal{L}^\e\), we can see that the projections onto \(i\chi_j w(z-\xi^+_j)\) are equal to zero. Indeed, if
\[
\L^\e(\phi)=iV_dh+iV_d\sum_{j=1}^{n^+} \left\{c^1_{j}\frac{\chi_j w_{x_1}(z-\xi^+_j)}{iw(z-\xi^+_j)} +c^2_{j}\frac{\chi_j w_{x_2}(z-\xi^+_j)}{iw(z-\xi^+_j)}+ c_j^3 \chi_j \right\}\]
since \( \L^\e(\phi)(\overline{z})=\overline{\L^\e(\phi)(z)}\) and \(\chi_j(\overline{z})=\chi_j(z)\) we find that
\[ \sum_{j=1}^{n^+} c_j^3\chi_j(z)= -\sum_{j=1}^{n^+} c_j^3\chi_j(z)\]
and then \(c_j^3=0\) for all \(j=1,\dots,n^+\).
Next we notice that, by \eqref{eq:eqsymmetriesofpsi}, the Lyapunov-Schmidt coefficients \(c^1_j,c^2_j\) \(1\leq j \leq n^+\) are all related and we can work with only one coefficient.

\begin{lemma}
Let \(\phi\) be a solution to
\[\L^\e(\phi)=iV_dh+iV_d\sum_{j=1}^{n^+} \left\{c^1_{j}\frac{\chi_j w_{x_1}(z-\xi_j^+)}{iw(z-\xi_j^+)}+c^2_{j}\frac{\chi_j w_{x_2}(z-\xi_j^+)}{iw(z-\xi_j^+)}\right\}, \quad \text{ with } \phi=iV_d\psi \text{ and } \|\psi\|_*<+\infty\]
and assume that $\psi$ and $h$ satisfy \eqref{eq:eqsymmetriesofpsi}.
Then, all the coefficients \(c_j^1,c_j^2\) can be expressed in terms of \(c_1^1\) only. 
\end{lemma}

\begin{proof}
Since \( V_d(e^{\frac{2i\pi}{n^+}}z)=V_d(z)\), and the cut-off function \(\eta\) defined in \eqref{def:cut_off_eta} also satisfies this property, it can be seen that 
$$\L^\e[\phi(e^{\frac{2i\pi}{n^+}}z)]=\L^\e(\phi)(e^{\frac{2i\pi}{n^+}}z).$$
Indeed we can write \(\L^\e(\phi)=\eta L_0(\phi)+(1-\eta)iV_dL'(\psi)\), with \(\phi=iV_d\psi\), \(L_0\) defined by \eqref{L0} and \(L'\) defined in \eqref{eq:mathL}. For the Laplacian part in the operator  \(L_0\) and \(L'\) it is well-known. The other terms involve the identity operator or are differential operators in the angular variable \(s\). Multiplying by \(e^{\frac{2i\pi}{n^+}}\) the variable \(z\) amounts to make a translation in the \(s\) variable, and hence the differential operators in \(s\) respect the symmetry.

Furthermore, noticing that 
$$e^{\frac{2i\pi}{n^+}}z-\xi_j^+=(z-\xi_{j-1}^+)e^{\frac{2i\pi}{n^+}},$$
we have  \(w(e^{\frac{2i\pi}{n^+}}z-\xi_j^+)=e^{\frac{2i\pi}{n^+}}w(z-\xi_{j-1}^+)\) and, using formulae \eqref{derW},
\begin{equation*}\begin{split}
w_{x_1}(e^{\frac{2i\pi}{n^+}}z-\xi_j^+) &=e^{\frac{2i\pi}{n}}\left(\cos \left(\frac{2\pi}{n^+}\right) w_{x_1}(z-\xi_{j-1}^+)-\sin\left(\frac{2\pi}{n^+}\right)w_{x_2}(z-\xi_{j-1}^+)\right), \nonumber \\
w_{x_2}(e^{\frac{2i\pi}{n^+}}z-\xi_j^+) &=e^{\frac{2i\pi}{n}}\left(\sin \left(\frac{2\pi}{n^+}\right)w_{x_1}(z-\xi_{j-1}^+)+\cos\left(\frac{2\pi}{n^+}\right)w_{x_2}(z-\xi_{j-1}^+)\right), \nonumber
\end{split}\end{equation*}
where the indices are taken modulo \(n^+-1\).
 Since \(\psi( e^{\frac{2i\pi}{n^+}}z)=\psi(z)\), we must have
 \begin{equation}
 \sum_{j=1}^{n^+} \left\{c^1_{j}\frac{\chi_j w_{x_1}(e^{\frac{2i\pi}{n^+}}z-\xi_j^+)}{iw(e^{\frac{2i\pi}{n^+}}z-\xi_j^+)}+c^2_{j}\frac{\chi_j w_{x_2}(e^{\frac{2i\pi}{n^+}}z-\xi_j^+)}{iw(e^{\frac{2i\pi}{n^+}}z-\xi_j^+)}\right\}=\sum_{j=1}^{n^+} \left\{c^1_{j}\frac{\chi_j w_{x_1}(z-\xi_j^+)}{iw(z-\xi_j^+)}+c^2_{j}\frac{\chi_j w_{x_2}(z-\xi_j^+)}{iw(z-\xi_j^+)}\right\} 
 \end{equation}
and this implies that
\begin{equation*}
\begin{pmatrix}
\cos \left(\frac{2\pi}{n^+}\right) & \sin \left(\frac{2\pi}{n^+}\right) \\
-\sin \left(\frac{2\pi}{n^+}\right) & \cos \left(\frac{2\pi}{n^+}\right)
\end{pmatrix}
\begin{pmatrix}
c^1_j \\
c^2_j
\end{pmatrix}
=\begin{pmatrix}
c^1_{j-1} \\
c^2_{j-1}.
\end{pmatrix}
\end{equation*}
Hence all the coefficients can be expressed in terms of \(c_1^1,c_1^2\). But now we can use the symmetry with respect to the horizontal axis. It can be seen that
$$\mathcal{L}^\e(\phi)(\overline{z})=\overline{\mathcal{L}^\e(\phi)(z)},$$
and thus, since \( \phi(\overline{z})=\overline{\phi(z)}\), we have
\begin{equation*}
\sum_{j=1}^{n^+} \left\{c^1_{j}\frac{\chi_j w_{x_1}(\overline{z}-\xi_j^+)}{iw(\overline{z}-\xi_j^+)}+c^2_{j}\frac{\chi_j w_{x_2}(\overline{z}-\xi_j^+)}{iw(\overline{z}-\xi_j^+)}\right\}=\sum_{j=1}^{n^+} \left\{c^1_{j}\frac{\chi_j \overline{w}_{x_1}(z-\xi_j^+)}{i\overline{w}(z-\xi_j^+)}+c^2_{j}\frac{\chi_j \overline{w}_{x_2}(z-\xi_j^+)}{i\overline{w}(z-\xi_j^+)}\right\}.
\end{equation*}
Using that \(w_{x_1}(\overline{z})=\overline{w}_{x_1}(z)\) and \(w_{x_2}(\overline{z})=-\overline{w}_{x_2}(z)\), we conclude that necessarily \(c^2_1=0\). Thus all the coefficients can be expressed in terms of \(c^1_1\) only.
\end{proof}
\begin{remark}
For Theorem \ref{th:main2} we also need to consider the vortex of degree $-1$ at the origin, corresponding to the case $n^-=1$ and $\xi_1^-=0$. That is the right hand-side of \eqref{eq:linear} has to be modified to
\begin{equation*}
iV_dh+iV_d \Bigl[c^0_1\frac{\chi\overline{w}_{x_1}(z)}{iw(z)}+c^0_2 \frac{\chi\overline{w}_{x_2}(z)}{iw(z)}+\sum_{j=1}^{n^+}  \left\{c^1_{j}\frac{\chi_j w_{x_1}(z-\xi_j^+)}{iw(z-\xi_j^+)}+c^2_{j}\frac{\chi_j w_{x_2}(z-\xi_j^+)}{iw(z-\xi_j^+)}\right\}\Bigr]
\end{equation*}
As in the previous case, if $\psi$ satisfies \eqref{eq:eqsymmetriesofpsi} it follows that \(c^1_j,c^2_j\) can be expressed in terms of \(c^1_1\) only. By using the symmetry \(\psi(e^{\frac{2i\pi}{n^+}}z)=\psi(z)\) we can also see that \(c^0_1=c^0_2=0\).
\end{remark}

Recall the notation given in \eqref{xij}. Writing $\psi:\mathbb C\to\mathbb C$ as $\psi=\psi_1+i\psi_2$ we define, given $\alpha , \sigma \in (0,1)$,
\begin{align*}
\|\psi\|_* := \sum_{j=1}^{N}  \| V_d \psi \|_{C^{2,\alpha}(r_j<3)} + \| \psi_1 \|_{1,*} + \| \psi_2 \|_{2,*},
\end{align*}
where
\begin{align*}
\|\psi_1\|_{1,*} 
&:= \sum_{j=1}^{N} \|\psi_1\|_{L^\infty(r_j>2)}+\sup_{2<r_j<\frac{2}{\e},\,1\leq j\leq N}|\nabla \psi_1|\left(\sum_{j=1}^Nr_j^{-1}\right)^{-1}+\sup_{r>\frac{1}{\e}}\left(\frac{1}{\e}|\p_r\psi_1|+|\p_s\psi_1|\right)\\
&+ \sup_{2<r_j<R_\e,\,1\leq j\leq N}|D^2\psi_1|\left(\sum_{j=1}^Nr_j^{-2}\right)^{-1} \\
&+\sup_{2<|z-\xi_j|<R_\e,\,1\leq j\leq N}[D^2\psi_1]_{\alpha,B_{|z|/2}(z)}\left(\sum_{j=1}^N|z-\xi_j|^{-2-\alpha}\right)^{-1},
\\
\|\psi_2\|_{2,*} 
&:= \sup_{r_j>2,\,1\leq j\leq N}|\psi_2|\left(\sum_{j=1}^Nr_j^{-2+\sigma}+\e^{\sigma-2}\right)^{-1}+\sup_{2<r_j<\frac{2}{\e},\,1\leq j\leq N}|\nabla \psi_2|\left(\sum_{j=1}^Nr_j^{-2+\sigma}\right)^{-1}\\
&+\sup_{r>\frac{1}{\e}}\left(\frac{1}{\e^{2-\sigma}}|\p_r \psi_2|+\frac{1}{\e^{1-\sigma}}|\p_s\psi_2|\right)+ \sup_{2<r_j<R_\e,\,1\leq j\leq N}|D^2\psi_2|\left(\sum_{j=1}^Nr_j^{-2+\sigma}\right)^{-1}\\
&+\sup_{2<|z-\xi_j|<R_\e,\,1\leq j\leq N}[D^2\psi_2]_{\alpha,B_{|z|/2}(z)}\left(\sum_{j=1}^N|z-\xi_j|^{-2+\alpha}\right)^{-1}.
\end{align*}
We also recall the definition of the norm $\|\cdot\|_{**}$ given in \eqref{def:norm_**}. Indeed, given a control on the right hand size measured with $\|\cdot\|_{**}$, the norm $\|\cdot\|_{*}$ provides the best decay we can expect for the solution $\psi$ (for both real and imaginary parts) and its derivatives, as the following proposition states. Thus we can establish the following invertibility result for problem \eqref{eq:linear}.
\begin{proposition}\label{prop:linearfull}
If $h$ satisfies \eqref{eq:eqsymmetriesofpsi} and $\|h\|_{**}<+\infty$ then for $\e>0$ sufficiently small there exists a unique solution $\phi=T_\e(iV_dh)$ to \eqref{eq:linear} with $\|\psi\|_*<\infty$, where \(\phi=iV_d \psi\).
Furthermore, there exists a constant $C>0$ depending only on  $\alpha, \sigma \in(0,1)$ such that this solution satisfies
\begin{equation*}
\|\psi\|_*\leq C\|h\|_{**}.
\end{equation*}
\end{proposition}

The proof can be found at subsection \ref{sec:prop1}. This result allows us to solve a non-linear projected problem, following the usual scheme of the Lyapunov-Schmidt reduction methods. However, the a priori estimate in Proposition \ref{prop:linearfull} is not enough to solve the reduced problem. More precisely, due to the large size of the error of the approximation in the norm $\|\cdot\|_{**}$ (which is of order $|\log\e|^{-1}$, see Proposition \eqref{globalError}), a fixed point argument would give a too large $\psi$, making impossible to choose the parameter \(\hat{d}_\e\) so that the Lyapunov-Schmidt coefficient \(c_1^1\) in \eqref{eq:linear} vanishes.

To overcome this difficulty we will need more accurate a priori estimates, relying on the symmetries of the error and the function $\psi$. Indeed, the largest part of the error can be seen to be orthogonal to the kernel (see Lemma \ref{decError}) and it will not play a role at the reduction step, what allows us to refine the estimates according to its symmetry, in the spirit of Lemma \ref{lemma:decomposition}.

Let us consider $\psi:\mathbb C\to \mathbb C$ and the relation $
z = \rho_j e^{i\theta_j} + \xi_j $.
We  can decompose \(\psi\) in Fourier series in $\theta_j$ as in \eqref{def:h_Fourier}
and define
\begin{align*}
\psi^{e,j} := \sum_{k \text{ even}} \psi^{k,j}, \quad
\psi^{o,j} := \sum_{k \text{ odd}} \psi^{k,j}.
\end{align*}
The idea behind making this decomposition is that $\psi^{e,j}$ is large but orthogonal to the kernel near $\xi_j$ by symmetry, while $\psi^{o,j}$ is not orthogonal but small.
With \(\mathcal{R}_j\) given in \eqref{def:reflection_R_j}, we have
\begin{align}
\nonumber
\psi^{o,j} (\mathcal{R}_j z ) &=  \overline{\psi^{o,j}(z)} ,\quad 
\psi^{e,j}(\mathcal{R}_j z ) = - \overline{\psi^{e,j}(z)} ,
\end{align}
and we can define equivalently
\begin{align}
\label{psiejoj}
\psi^{o,j}(z) &= \frac{1}{2}[ \psi(z) + \overline{\psi(\mathcal{R}_j z )}] , \quad
\psi^{e,j}(z) = \frac{1}{2}[ \psi(z) - \overline{\psi(\mathcal{R}_j z )}] .
\end{align}
Let \(R_\e\) and \(\eta_{j,R}\) from \eqref{Reps} and  \eqref{def:etajR}. We consider a global function  $\psi^o$ defined as  
\begin{align}
\label{def-psio}
 \psi^o &:= \sum_{j=1}^N\eta_{j,\frac{R_\e}{2}} \psi^{o,j},
\end{align}
that represents the {\em  odd} part of $\psi$ around each vortex $\xi_j$, localized with a cut-off function, and corresponds to the small part of $\psi$.

This part arises from terms in the error $R^o$ that are small, but decay slowly, so we need to estimate it in norms that allow for growth up to a certain distance. Namely, 
\begin{align*}
| \psi |_\sharp &:= 
\sum_{j=1}^N
|\log\varepsilon|^{-1}
\| V_d \psi \|_{C^{2,\alpha}(r_j<3)}
+ |\psi_1|_{\sharp,1}+ |\psi_2|_{\sharp,2} ,
\end{align*}
where
\begin{align}
\label{normSharp1}
|\psi_1|_{\sharp,1}
& :=
\sup_{2<r_j<R_\e,\, 1\leq j\leq N}\left[|\psi_1|\left(\sum_{j=1}^Nr_j\log(2R_\e/r_j)\right)^{-1}+|\nabla\psi_1|\left(\sum_{j=1}^N\log(2R_\e/r_j)\right)^{-1}\right],
\\
\label{normSharp2}
|\psi_2|_{\sharp,2} 
&:=
\sup_{2<r_j<R_\e,\, 1\leq j\leq N}\left[(|\psi_2|+|\nabla\psi_2|)\left(\sum_{j=1}^N(r_j^{-1+\sigma}+r_j^{-1}\log(2R_\e/r_j))\right)^{-1}\right],
\end{align}
with $\sigma\in (0,1)$. The norm \(|\cdot|_{\sharp }\) is built in correspondence with the norm \(|\cdot|_{\sharp \sharp}\) which estimates the {\it odd} part of the error of the ansatz. With the help of this norm we can establish precise estimates on the {\it odd} part of $\psi$.

\begin{proposition}
\label{prop:sharp2b}
Suppose that $h$ satisfies the symmetries \eqref{eq:eqsymmetriesofpsi}
and  $\|h \|_{**}<\infty$.
Suppose furthermore that  $h^o $ defined by \eqref{def-ho} is decomposed as $ h^o = \hat{h}^o + \tilde{h}^o$ where $|  \hat{h}^o |_{\sharp\sharp}<\infty$ 
and $\hat{h}^o$, $\tilde{h}^o$ satisfy 
\begin{align}
\nonumber
\hat{h}^o(\mathcal{R}_j z ) =  \overline{\hat h^o(z)},\quad \tilde{h}^o(\mathcal{R}_j z ) =  \overline{\tilde h^o(z)},\quad 
|z-\xi_j| <  R_\varepsilon, \quad j=1,\cdots,N,
\end{align}
and have support in $\bigcup_{j=1}^NB_{2R_\varepsilon}(\xi_j) $.
Let us write $\psi = \psi^e + \psi^o$ with $\psi^o$ defined by \eqref{def-psio}. Then there exists $C>0$ such that $\psi^o$ can be decomposed as $\psi^o = \hat \psi^o + \tilde \psi^o$, with each function supported in 
$\bigcup_{j=1}^N B_{2R_\varepsilon}(\xi_j) $ 
and satisfying 
\begin{align}
\label{est:prop5.3-1}
|\hat\psi^o|_\sharp 
& \leq C\left(
|\hat h^o |_{\sharp\sharp} 
+ \varepsilon \sqrt{ | \log\varepsilon|} ( \|\hat h^o\|_{**}  
+  
\|  h - h^o\|_{**} )\right)
\\
\label{est:prop5.3-2}
\| \tilde\psi^o \|_* 
&\leq C
 \|\tilde h^o\|_{**},  
 \\
 \|\hat \psi^o\|_*+\|\tilde \psi^o\|_* & \leq C\left( \|h\|_{**}+\|\hat h^o\|_{**}+\|\tilde h^o\|_{**}\right) \nonumber
\end{align}

and
\begin{align}
\nonumber
\hat{\psi}^o(\mathcal{R}_j z ) =  \overline{\hat \psi^o(z)},\quad \tilde{\psi}^o(\mathcal{R}_j z ) =  \overline{\tilde \psi^o(z)},\quad 
|z-\xi_j| <  R_\varepsilon, \quad j=1,\cdots,N.
\end{align}
\end{proposition}

\subsection{First a priori estimate and proof of Proposition~\ref{prop:linearfull}}
\label{sec:prop1}
In this section, our aim is to solve the linear projected problem \eqref{eq:linear}. We first obtain a priori estimates and then use these estimates and the Fredholm alternative to obtain the solution. We first deal with the following problem:
\begin{align}
\label{eq:linearhomogeneous}
\left\{
\begin{aligned}
& \L^\e(\phi)=iV_d h \text{ in } \R^2 \\
& \RE\int_{B(0,4)} \chi_j\overline{\phi_j}w_{x_1}=0, \text{ with }\phi_j(z)=iw(z)\frac{\phi(z+\xi_j^+)}{iV_d(z+\xi_j^+)}, \ j=1,\cdots,n^+,\\
& \psi=\frac{\phi}{iV_d} \text{ satisfies the symmetry }  \eqref{eq:eqsymmetriesofpsi}.
\end{aligned}
\right.
\end{align}
\begin{lemma}\label{FirstEstimate}
There exists a constant $C>0$  such that for all $\e$ sufficiently small and any solution $\phi=iV_d\psi$ of \eqref{eq:linearhomogeneous} with $\|\psi\|_*<\infty$ one has
\begin{align}
\label{est0a}
\|\psi\|_*\leq C\|h\|_{**}.
\end{align}
\end{lemma}

\begin{proof}[Proof of Lemma~\ref{FirstEstimate} ]
The proof follows as in \cite[Lemma 5.1]{DdPMR} by using barrier arguments, so we will only highlight the differences. 

Near the vortices the argument remains essentially the same as a consequence of Lemma \ref{lem:ellipticestimatesL0}. Far from them, in the region $\bigcap_{j=1}^{n^+}\{r_j>2\}$, the function $\psi=\psi_1+i\psi_2$ solves 
\begin{align*}
h=
\Delta\psi  + 2 \frac{\nabla V_d\nabla \psi}{V_d}
-2 i |V_d |^2 \psi_2
+\varepsilon^2 \p^2_{ss}\psi
+ 
\varepsilon^2 
\Bigl(
2\frac{\partial_{s}V_d}{V_d}
-2in
\Bigr)
\p_s\psi +ic \e^2 |\log \e| \p_s \psi,
\end{align*}
and thus we only need to deal with the new term
$$ic \e^2|\log\e|\p_s\psi.$$
It can be seen that, for $R$ large and some $1\leq j\leq N$,
\begin{equation*}\begin{split}
|c \e^2 |\log \e| \p_{s} \psi_1| & \leq C (R^{-\sigma'}+\e^{\sigma'})\left(\frac{1}{r_j^{2-\sigma}}+\e^{2-\sigma}\right)\|\psi_1\|_{1,*,0},\\
|c \e^2 |\log\e|\p_{s}\psi_2|&\leq C\left(\e^{1-\sigma''}+R^{\sigma''-1}\right)\left(\frac{1}{r_j^2+\e^2}\right)\|\psi_2\|_{2,*,0},
\end{split}\end{equation*}
for some \(0<\sigma'<\sigma<1\), \(0<\sigma''<1\), where
\begin{align*}
\|\psi_1\|_{1,*,0} 
:=&\, 
\sum_{j=1}^N \|\psi_1\|_{L^\infty(r_j>2)}+\sup_{2< r_j<\frac{2}{\e}, 1\leq j\leq N}|\nabla\psi_1|\left(\sum_{j=1}^Nr_j^{-1}\right)^{-1}+\sup_{r>\frac{1}{\e}}\left(\frac{1}{\e}|\p_r\psi_1|+|\p_s\psi_1|\right)
\\
\|\psi_2\|_{2,*,0} 
:=&\,\sup_{r_j>2,\,1\leq j\leq N}|\psi_2|\left(\sum_{j=1}^N r_j^{-2+\sigma}\right)^{-1}+\sup_{2< r_j<\frac{2}{\e}, 1\leq j\leq N}|\nabla\psi_2|\left(\sum_{j=1}^Nr_j^{-2+\sigma}\right)^{-1}\\
&+\sup_{r>\frac{1}{\e}}\left(\e^{\sigma-2}|\p_r\psi_1|+\e^{\sigma-1}|\p_s\psi_1|\right),
\end{align*}
for some $0<\sigma<1$. Thus the result follows by comparison arguments choosing respectively the barriers 
\begin{equation*}\begin{split}
\mathcal B_1 &:= M_1 \theta_1 \left(\pi-\frac{n\theta_1}{2}\right),\qquad M_1 := C\left( \|h \|_{**,0} + \e^{1-\sigma''}+R^{\sigma''-1}+ \|\psi_1\|_{L^\infty(B_R(\xi_j))}\right),\\
\mathcal{B}_2 &:=M_2 \left( \frac{1}{r_j^{2-\sigma}}+\e^{2-\sigma} \right),\qquad M_2:=C\left(\|h\|_{**,0}+R^{-\sigma'}+\varepsilon^{\sigma'} +\|\psi_2\|_{L^\infty(B_R(\xi_j))} \right),
\end{split}\end{equation*}
with $C$ a large fixed constant and
\begin{equation*}\begin{split}
\|h\|_{**,0}:=&\,\sum_{j=1}^N \| V_d h\|_{L^\infty(r_j<3)}
+ \sup_{r_j>2,\,1\leq j\leq N}|\RE(h)|\left(\sum_{j=1}^Nr_j^{-2}+\e^2\right)^{-1}\\
&+|\IM(h)|\left(\sum_{j=1}^Nr_j^{-2+\sigma}+\e^{2-\sigma}\right)^{-1},
\end{split}\end{equation*}
for $0<\sigma<1$.
\end{proof}

\begin{proof}[Proof of Proposition~ \ref{prop:linearfull}]
The result follows as a consequence of the Riesz representation theorem and the Fredholm alternative proceeding as in \cite[Proposition 5.1]{DdPMR}, that is, rewriting the problem as 
\begin{equation*}
[\phi,\varphi]-\langle k(x)\phi,\varphi \rangle =\langle s,\varphi\rangle, \quad\forall \varphi \in \mathcal{H},
\end{equation*}
where $\mathcal{H}$ is the Hilbert space
\begin{equation*}\begin{split}
\mathcal{H}:=& \Bigl\{\phi=iV_d\psi \in H^1_0(B_M(0),\mathbb{C}); \;\RE \int_{B(0,4)} \chi\overline{\phi}_jw_{x_1}=0, \ j=1,\cdots,n^+, \;\;\psi \text{ satisfies \eqref{eq:eqsymmetriesofpsi}}\Bigr\},
\end{split}\end{equation*}
for $M>10|\xi_1|$, equipped with the inner product
\begin{equation*}
[\phi,\varphi]:= \RE \int_{B(0,M)} \left( \nabla \phi \overline{\nabla\varphi} +\e^2\p_s \phi \overline{\p_s \varphi} \right).
\end{equation*}
Here $\phi_j(z):=iw(z)\psi(z+\xi_j)$ and $\xi_j(z):=\eta_1\left(\frac{|z-\xi_j|}{2}\right)$.
Using \eqref{defL_j}, $\langle k(x)\phi,\cdot\rangle$, $\langle s,\cdot\rangle$ correspond to the linear forms 
\begin{equation*}\begin{split}
\langle k(x)\phi,\varphi \rangle:=& \,\e^2\RE\int_{B(0,M)}\left(2ni\phi\overline{\p_s\varphi}-n^2\phi \overline{\varphi} \right)-2\RE  \int_{B(0,M)} \RE(\overline{\phi}V_d)V_d \overline{\varphi}  \\
&\,+\RE \int_{B(0,M)} [(\eta-1)\frac{E}{V_d}+(1-|V_d|^2)]\phi \overline{\varphi}+c \e^2|\log\e|\RE \int_{B(0,M)} (in\phi\overline{\varphi}-\phi\overline{\p_s\varphi}),\\
\langle s,\varphi \rangle:= &\RE \int_{B(0,M)}\left(h+iV_d\sum_{j=1}^{n^+}c_j\chi_j\frac{w_{x_1}(z-\xi_j)}{iw(z-\xi_j)}\right) \overline{\varphi}.
\end{split}\end{equation*}
defined on $\mathcal{H}$.

The rest of the proof follows as in \cite[Proposition 5.1]{DdPMR}.
\end{proof}

\subsection{Second a priori estimate and proof of Proposition \ref{prop:sharp2b}}
\label{sec:prop2}

\begin{lemma}
\label{lemma:aprioriSharp}
Let $\alpha \in (0,1)$, $\sigma \in (0,1)$.
Then  there exists a constant $C>0$  such that for all $\e$ sufficiently small and any solution  $\phi$ of \eqref{eq:linearhomogeneous} with $\phi=iV_d\psi$ and $\|\psi\|_*<\infty$  one has
\begin{equation}
\label{claimSharp}
|\psi|_\sharp \leq C(  |h |_{\sharp\sharp} + \varepsilon \sqrt{ | \log\varepsilon|} \|h\|_{**} )  .
\end{equation}
\end{lemma}
\begin{proof}
The result follows as a consequence of Lemma \ref{FirstEstimate} and a barrier argument. Indeed, writting $\psi=\psi_1+i\psi_2$, it can be checked that, for some $1\leq j\leq N$,
\begin{equation*}\begin{split}
|c \e^2 |\log \e| \p_{s} \psi_1| &\leq \frac{C}{r_j}\log \left(\frac{2R_\e}{r_j}\right)|\psi_1|_{\sharp,1},
\end{split}\end{equation*}
and consequently, proceeding as in \cite[Lemma 5.2]{DdPMR},
$$|\nabla \psi_2|\leq \tilde{\mathcal B}_2,$$
with 
\begin{align*}
\tilde{\mathcal{B}}_2 := 
\frac{C}{r_j^{1-\sigma}} ( 
|h|_{\sharp\sharp,0} 
+\|\psi_2\|_{L^\infty(r_j=R_{\varepsilon})}  )
+ \frac{C}{r_j}\log\Bigl(\frac{2R_\varepsilon}{r_j}\Bigr)
\left(  |\psi_1|_{\sharp,1} 
+
\frac{\|\psi_2\|_{L^\infty(r_j=R_0)} }{ |\log\varepsilon |}
\right),
\end{align*}
with $R_0>0$ a large fixed constant, $|\ |_{\sharp,1}$ defined in \eqref{normSharp1}
and, denoting $h=h_1+ih_2$,
\begin{align*}
| h |_{\sharp \sharp,0}
&:=
\sum_{j=1}^N  \| V_d h\|_{L^\infty(r_j<4)}
+\sup_{2<r_j<R_\e,\,1\leq j\leq N}\left[|h_1|\left(\sum_{j=1}^Nr_j^{-1}\right)^{-1}+|h_2|\left(\sum_{j=1}^Nr_j^{-1+\sigma}\right)^{-1}\right],
\end{align*}
$0<\sigma<1$. Therefore, 
\begin{equation*}\begin{split}
|c \e^2|\log \e| \p_{s} \psi_2| &\leq 
\frac{C}{r_j|\log\e|}
\tilde{\mathcal B}_2,
\end{split}\end{equation*}
and the result follows as a straightforward adaptation of \cite[Lemma 5.2]{DdPMR}.
\end{proof}

Before proving Proposition~\ref{prop:sharp2b} we consider the  solution constructed in Proposition~\ref{prop:linearfull} when the right hand side has symmetries. More precisely, 
let us consider the local symmetry condition
\begin{align}
\label{symmetryhe}
h(\mathcal{R}_j z ) = - \overline{h(z)}  ,\quad 
|z-\xi_j| < 2 R_\varepsilon , \quad j=1,\cdots,N,
\end{align}
where $\mathcal{R}_j $ was defined in \eqref{def:reflection_R_j}.

\begin{lemma}
\label{lemma:sharp-even}
Suppose that $h$ satisfies the symmetries \eqref{eq:eqsymmetriesofpsi} and \eqref{symmetryhe}.
We assume that
\[
\| h \|_{**} <\infty.
\]
Then there exist  $\psi^s $, $ \psi^*$ such that  
$\psi=\frac{\phi}{iV_d}$ with $\phi$ the solution to \eqref{eq:linear} and $\|\psi\|_*<\infty$
can be written as $\psi = \psi^s + \psi^*$ with the estimates
\begin{align*}
\| \psi^s \|_* + \| \psi^* \|_* &\leq C  \| h \|_{**} 
\\
| \psi^* |_\sharp &\leq C  \varepsilon \sqrt{ |\log\varepsilon|}   \| h \|_{**} .
\end{align*}
Moreover $(\psi^s, \psi^*)$ define  linear operators of $h$,  $\psi^s$ has its
support in $\bigcup_{j=1}^N B_{R_\varepsilon}(\xi_j)$ 
and satisfies
\begin{align}
\label{symmetryPsyE}
\psi^s(\mathcal{R}_j z ) &= - \overline{\psi^s(z)} ,\quad 
|z-\xi_j| <  R_\varepsilon, \quad 1\leq j\leq N.
\end{align}
\end{lemma}
\begin{proof}
The proof follows analogously to \cite[Lemma 5.3]{DdPMR} by splitting $\mathcal{L}^\e$ into a part $\mathcal{L}^\e_{s,j}$ that preserves the symmetry \eqref{symmetryPsyE} and a remainder term $\mathcal{L}^\e_{r,j}$, for every $1\leq j\leq N$. Indeed, we consider the linear operators
\begin{align*}
L'_{s,j}(\psi)
&:=
\Delta \psi 
+ 2\frac{\nabla w(z-\xi_j) \nabla \psi}{w(z-\xi_j)}
-2i |w(z-\xi_j)|^2 \IM(\psi)
\\
&
+\varepsilon^2 \Biggl[ d_{\e}^2 \partial_{r_j r_j}^2 \psi \sin^2(\theta_j-\varphi_j)
+d_{\e}^2 \cos(\theta_j-\varphi_j)\sin(\theta_j-\varphi_j)\left(\frac{2\p^2_{r_j\theta_j}\psi}{r_j}-\frac{2\p_{\theta_j}\psi}{r_j^2} \right)
\\
& 
+\partial_{\theta_j \theta_j}^2 \psi \Bigl(1 + \frac{d_{\e}^2}{r_j^2} \cos^2(\theta_j-\varphi_j) \Bigr)
+ \partial_{r_j} \psi  \frac{d_{\e}^2}{r_j} \cos^2(\theta_j-\varphi_j) 
-2 \partial_{\theta_j} \psi 
\frac{d_{\e}^2}{r_j^2} \sin(\theta_j-\varphi_j)\cos(\theta_j-\varphi_j) 
\Biggr],
\end{align*}
and
\begin{align*}
L'_{r,j}&(\psi)
:=2\sum_{l\neq j }\frac{\nabla w(z-\xi_l) \nabla \psi}{w(z-\xi_l)}
-2 i ( |V_d|^2 -  |w(z-\xi_j)|^2 )  \IM(\psi) 
\\
& 
+ \varepsilon^2
\Biggl[
2 d_{\e}
\partial_{r_j \theta_j}^2 \psi  
 \sin(\theta_j-\varphi_j) 
+2 \partial_{ \theta_j\theta_j}^2 \psi  \frac{d_{\e}}{r_j} \cos(\theta_j-\varphi_j) 
+ \partial_{r_j} \psi  d_{\e} \cos(\theta_j-\varphi_j)
-  \partial_{\theta_j} \psi
\frac{d_{\e}}{r_j} \sin(\theta_j-\varphi_j) 
\Biggr]
\\
& 
+ 
\varepsilon^2 \Bigl( \frac{2\partial_s V_d }{V_d}
- 2ni  +ic_\varepsilon|\log\e|\Bigr) 
\Bigl[  \partial_{r_j} \psi d_{\e} \sin(\theta_j-\varphi_j)
+  \Bigl(1 + \frac{d_{\e}}{r_j} \cos(\theta_j-\varphi_j) \Bigr)\partial_{\theta_j}\psi   \Bigr].
\end{align*}
We also define 
\begin{equation*}
L_{0,s,j}(\phi):=iV_dL'_{s,j}(\psi)+i(E-E^o)\psi, \quad L_{0,r,j}(\phi):=L_0(\phi)-L_{0,s,j}(\phi)
\end{equation*}
where \(E^o\) is defined analogously to \eqref{def-psio}. We then set
\begin{equation*}
\L^\e_{s,j}(\phi):=\eta L_{0,s,j}(\phi) +(1-\eta)iV_d L'_{s,j}(\psi), \quad \phi=iV_d \psi,
\end{equation*}
\begin{equation*}
\L^\e_{r,j}(\phi):= \L^\e (\phi)- \L^\e_{s,j}(\phi).
\end{equation*}
The rest of the proof follows as in \cite[Lemma 5.3]{DdPMR} by applying Proposition \ref{prop:linearfull} and Lemma \ref{lemma:aprioriSharp}.
\end{proof}
As a consequence of these results we can conclude the statement of Proposition~\ref{prop:sharp2b}.
\begin{proof}[Proof of Proposition~\ref{prop:sharp2b}]
The result follows by putting together Proposition~\ref{prop:linearfull}, Lemma~\ref{lemma:aprioriSharp} and Lemma \ref{lemma:sharp-even}.
\end{proof}

Once we have established the solvability and the a priori estimates for the projected linear problem \eqref{eq:linear} we can handle the non linear case \eqref{eq:nonlinear}.

\begin{proposition}\label{prop:nonlinear}
There exists a constant $C>0$ depending only on $0<\alpha,\sigma<1$, such that for all $\e$ sufficiently small there exists a unique $\psi_\e$ such that \(\phi_\e=iV_d\psi_\e\) is the solution of \eqref{eq:nonlinear}, that satisfies
\begin{equation*}
\|\psi_\e\|_*\leq \frac{C}{|\log \e|}.
\end{equation*}
Furthermore $\psi_\e$ is a continuous function of the parameter $\hat{d}_\e:=\e\sqrt{|\log \e|}d_\e$
\begin{eqnarray}\label{estimatesPhi}
|\psi_\e^o|_{\sharp} \leq  C \e \sqrt{|\log \e|},
\end{eqnarray}
where $\psi_\e^o$ is defined according to \eqref{def-psio}.
\end{proposition}

The existence of solution is obtained by combining the linear theory with a fixed point argument performed in a precise set determined by the size of the error term $R$ and the a priori estimates on the symmetric and non symmetric part of the solution. Notice that the non linear term $\mathcal{N}(\psi)$ is exactly the same as in the case of the Ginzburg-Landau equation in \cite{DdPMR}. Thus, by applying Proposition \ref{prop:linearfull}, Proposition \ref{prop:sharp2b} and Lemma \ref{lemma:decomposition} the result follows exactly as in \cite[Proposition 6.1]{DdPMR} so we omit the proof.

\section{Solving the reduced problem: proofs of theorem \ref{th:main1} and theorem \ref{th:main2}}\label{VII}

 The function $\psi_\e$, with  \(\phi_\e=iV_d\psi_\e\) the solution of \eqref{eq:nonlinear} found in Proposition \ref{prop:nonlinear}, depends continuously on $\hat{d}_\e:=\e\sqrt{|\log \e|}d_\e$. We want to find $\d_\e$ such that the Lyapunov-Schmidt coefficient in \eqref{eq:nonlinear} satisfies $c_1=c_1(\d_\e)=0$. 

By symmetry we work only in the sector
$$\Theta_1:=\left\{z\in \mathbb{C}:\, z=re^{is},\; r>0,\; s\in \left[-\frac{\pi}{n^+},\frac{\pi}{n^+}\right]\right\}.$$
In the previous section we have found \(\psi_\e\) such that
\begin{equation}\label{eq:reduction_theta_1}
\begin{split}
\left[\mathcal{L}^\e(\phi_\e)+E-N(\phi_\e) \right](z+\xi_1^+) = c_1iV_d(z+\xi_1^+)\chi(z) \frac{ w_{x_1}(z)}{iw(z)} \text{ in } \Theta_1.
\end{split}\end{equation}
We recall that \(R_\e\) is defined in \eqref{Reps} and thus satisfies that  \(R_\e=o_\e\left ((\e\sqrt{ |\log \e|})^{-1}\right )\) but \(|\log R_\e| \sim |\log \e|\), and we set
\begin{equation*}
c_*:= \RE \int_{B(0,R_\e)} \chi |w_{x_1}|^2=\RE \int_{B(0,4)}\chi |w_{x_1}|^2,
\end{equation*}
and we remark that, thanks to the decay of \(w_{x_1}\), this quantity is of order $1$.
We multiply the equation \eqref{eq:reduction_theta_1} by \(\frac{\overline{V_d}(z+\xi_1^+)}{\overline{w}(z)} \overline{w}_{x_1}(z)\) and we observe that
\begin{equation}
\left|\frac{V_d(z+\xi_1^+)}{w(z)} \right|^2=1+O(\e^2) \quad \text{ in } \Theta_1.
\end{equation} We find that
\begin{align*}
c_1 c_*=&-\RE \int_{B(0,R_\e)} \frac{\overline{V_d}(z+\xi_1^+)}{\overline{w}(z)} E(z+\xi_1^+)\overline{w}_{x_1}(z)  + \RE \int_{B(0,R_\e)}\frac{\overline{V_d}(z+\xi_1^+)}{\overline{w}(z)}\L^\e(\phi_\e)(z+\xi_1^+) \overline{w}_{x_1}(z) \\ 
&-\RE \int_{B(0,R_\e)} \frac{\overline{V_d}(z+\xi_1^+)}{\overline{w}(z)} N(\phi_\e)(z+\xi_1^+)\overline{w}_{x_1}(z) +O(\e^2).
\end{align*}
We observe that 
\begin{align*}
\frac{\overline{V_d}(z+\xi_1^+)}{\overline{w}(z)}\L^\e(\phi_\e)(z+\xi_1^+)& = \frac{\overline{V_d}(z+\xi_1^+)}{\overline{w}(z)}\left(iV_dL'(\psi_\e)+i\eta E\psi_\e \right)(z+\xi_1^+)  \\
&= \frac{|V_d(z+\xi_1^+)|^2}{|w(z)|^2}iw(z) L'(\psi_\e)(z+\xi_1^+)+\frac{\overline{V_d}(z+\xi_1^+)}{\overline{w}(z)}i(\eta E\psi_\e)(z+\xi_1^+) \\
&= L_1^\e(\phi_1)+\frac{\overline{V_d}(z+\xi_1^+)}{\overline{w}(z)}i(\eta E\psi_\e)(z+\xi_1^+)+O(\e^2),
\end{align*}
 with \(L_1^\e\) defined in \eqref{defL_j} and \(\phi_1\) defined in \eqref{defphi_j}. Integrating by parts we find
\begin{equation*}\begin{split}
\RE \int_{B(0,R_\e)} &L_1^\e(\phi_1)  \overline{w}_{x_1} =\RE \int_{B(0,R_\e)} (L_1^\e-L_0)(\phi_1)  \overline{w}_{x_1} +\RE \int_{B(0,R_\e)} L_0(\phi_1)  \overline{w}_{x_1} \\
&=
\RE \int_{B(0,R_\e)} (L_1^\e-L_0)(\phi_1)  \overline{w}_{x_1} +\RE \int_{\p B(0,R_\e)} \left(\frac{\p \phi_1}{\p \nu}\overline{w}_{x_1}-\phi_1\frac{\p \overline{w}_{x_1}}{\p\nu}\right) + O(\e^2|\log\e|),
\end{split}\end{equation*}
with $L_0$ given in \eqref{L0}. Using \eqref{eq:defL_j} and \eqref{eq:estimatesonalpha_j} we can estimate 
\begin{equation*}
\left|\RE \int_{B(0,R_\e)} (L_1^\e-L_0)(\phi_1)  \overline{w}_{x_1} \right| \leq C \e \sqrt{|\log \e|}\|\psi\|_* \leq \frac{C \e}{\sqrt{|\log \e|}},
\end{equation*}
and, by Lemma \ref{lem:propertiesofrho},
\begin{equation*}
\left|\RE \int_{\p B(0,R_\e)} \left(\frac{\p \phi_1}{\p \nu}\overline{w}_{x_1}-\phi_1\frac{\p \overline{w}_{x_1}}{\p\nu}\right)\right|\leq \frac{C \e}{\sqrt{|\log \e|}}.
\end{equation*}
Therefore,
\begin{equation}\label{projL}
\left|\RE \int_{B(0,R_\e)} L_1^\e(\phi_1)  \overline{w}_{x_1}\right| \leq \frac{C \e}{\sqrt{|\log \e|}}, 
\end{equation}
and we also have
\begin{align*}
\RE \int_{B(0,R_\e)} \frac{\overline{V_d}(z+\xi_1^+)}{\overline{w}(z)}\L^\e(\phi_\e)(z+\xi_1^+)\overline{w}_{x_1}= O \left(\frac{\e}{\sqrt{|\log \e|}}\right)
\end{align*}
since \(|E \psi_\e|\leq C\e/\sqrt{|\log \e|}\).
Now we estimate the inner product of \(w_{x_1}\) and the non linear term. We use Lemma \ref{lem:formulation} to write
\begin{equation*}\begin{split}
-\int_{B(0,R_\e)} &\frac{\overline{V_d}(z+\xi_1^+)}{\overline{w}(z)}N(\phi_\e)(z+\xi_1^+) \overline{w}_{x_1}(z)=\int_{B(0,2)\setminus B(0,1)} \frac{\overline{V_d}(z+\xi_1^+)}{\overline{w}(z)}M(\phi_\e)(z+\xi_1^+) \overline{w}_{x_1}(z)\\
&+\int_{B(0,R_\e)\setminus B(0,1)} iw(z) \frac{|V_d(x+\xi_1^+)|^2}{|w(z)|^2}(1-\eta)\mathcal{N}(\psi_\e)(z+\xi_1^+)\overline{w}_{x_1}(z)
\end{split}\end{equation*}
where
\begin{align*}
\mathcal{N}(\psi)= i(\nabla \psi)^2+\e^2(\p_s\psi)^2+i (e^{-2\IM(\psi)}-1+2\IM(\psi)).
\end{align*}

\noindent We use the orthogonality of the Fourier modes to write
\begin{equation*}\begin{split}
\RE \int_{B(0,R_\e)\setminus B(0,1)} iw &\mathcal{N}(\psi_\e)(z+\xi_1^+) \overline{w}_{x_1}  = \RE \int_{B(0,R_\e)\setminus B(0,1)}iw \overline{w}_{x_1} \left(\mathcal{N}(\psi_\e) \right)^o \nonumber \\
&= \RE \int_{B(0,R_\e)\setminus B(0,1)} i\rho \left( \rho' \cos s-\frac{i \rho}{r}\sin s\right)\left[ (\mathcal{N}(\psi_\e))^o_1+i(\mathcal{N}(\psi_\e))_2^o \right] \nonumber \\
&=- \int_{B(0,R_\e)\setminus B(0,1)} \left( \rho \rho' \cos s (\mathcal{N}(\psi_\e))^o_2-\frac{\rho^2}{r}(\mathcal{N}(\psi_\e))^o_1 \sin s \right).
\end{split}\end{equation*}
Using that
\begin{align*}
|\left(\mathcal{N}(\psi_\e)\right)^o_2| &\leq |\left(\mathcal{N}(\psi_\e)\right)^o_2|_{\sharp \sharp} \leq C \| \psi_\e^e\|_{*}|\psi_\e^o|_{\sharp}+|\psi_\e^o|_\sharp^2 \leq C\e |\log \e|^{-1/2}, \\
|\left(\mathcal{N}(\psi_\e)\right)^o_1| & \leq C \left( \frac{|(\psi_\e)_2^o|_\sharp \|(\psi_\e)_1^e\|_*}{1+r^2} +\frac{|(\psi_\e)_1^o|_\sharp \|(\psi_\e)_2\|_*}{1+r^{2-\sigma}}+\frac{|(\psi_\e)_1^o|_\sharp |(\psi_\e)_2^o|_\sharp}{1+r^{2-\sigma}} \right)\leq C \frac{\e |\log\e|^{-1/2}}{1+r^{2-\sigma}},
\end{align*}
we obtain
\begin{equation}\label{projN}
\left|\RE \int_{B(0,R_\e)\setminus B(0,1)} iw(z) \mathcal{N}(\psi_\e)(z+\xi_1^+) \overline{w}_{x_1} \right| \leq \frac{C\e}{\sqrt{|\log \e|}}.
\end{equation}
By using that \(M(\phi_\e)\) is at least quadratic and is a sum of analytic terms (multiplied by cut-off functions) in \(\phi_\e\) and \(\nabla \phi_\e\) we can use a parity argument analogous to the previous one to conclude that
\begin{equation*}
M(\phi_\e)^o =M(\phi_\e^o)+O\left(\|\psi_\e^e\|_*|\psi_\e^o|_\sharp\right)
\end{equation*}
and thus
\begin{equation*}
\left|\RE\int_{B(0,R_\e)} \frac{\overline{V_d}(z+\xi_1^+)}{\overline{w}(z)}N(\phi_\e)(z+\xi_1^+) \overline{w}_{x_1}(z) \right| \leq C \|\psi_\e^e\|_*|\psi_\e^o|_{\sharp} \leq  \frac{C\e}{\sqrt{|\log \e|}}.
\end{equation*}

It remains to estimate the term relative to the error. In order to do that we write \(E=iV_dR\) thus 
\begin{align*}
\int_{B(0,R_\e)}\frac{\overline{V_d}(z+\xi_1^+)}{\overline{w}(z)} E(z+\xi_1^+)\overline{w}_{x_1}(z) &= \int_{B(0,R_\e)} iw(z) R(z+\xi_1^+)\overline{w} _{x_1}(z) (1+O(\e^2)).
\end{align*}
We set
\begin{eqnarray}\label{defB}
B_a:= \RE \int_{B(0,R_\e)} iw(z) R_a(z+\xi_1^+)\overline{w}_{x_1}, \quad B_b:= \RE \int_{B(0,R_\e)} iw(z) R_b(z+\xi_1^+) \overline{w}_{x_1},
\end{eqnarray}
\begin{equation}\label{defB2}
B_c:=\RE \int_{B(0,R_\e)} iw(z) R_c(z+\xi_1^+)\overline{w}_{x_1},
\end{equation}
where we recall that $S_a(V_d)=iV_d R_a$, $S_b(V_d)=iV_dR_b$, $S_c(V_d)=iV_dR_c$ were given by \eqref{S}. 

{\it Proof of Theorem \ref{th:main1}.}  Assume $n^+=n\geq 2$ and $n^-=0$ in \eqref{eq:approx}. From Lemma \ref{decError} we find that
\begin{equation*}\begin{split}
B_b 
&=\frac{\d_\e \e}{\sqrt{|\log \e|}}\RE \int_{B(0,R_\e)}|w_{x_1}|^2 +O\left(\frac{\e}{{\sqrt{|\log \e|}}}\right)
\end{split}\end{equation*}
where we  used that \(\varphi_1=0\). We set
\begin{equation*}
a_1:=\frac{1}{|\log \e|}\int_0^{2\pi}\int_0^{R_\e} \frac{\r^2\sin^2s}{r} \dif r\dif s,
\end{equation*}
and we recall that \(|\log R_\e|\)  is of same order as \(|\log \e|\) and  does not depend on \(\d_\e\). Thus, using the fact that $\lim_{r\rightarrow +\infty}\r(r)=1$ we can see that 
\begin{equation}
a_1=\pi+o_\e(1).
\end{equation} 
Therefore we conclude that
\begin{equation}
B_b= \d_\e\pi \e \sqrt{|\log \e|} +o_\e(\e \sqrt{| \log \e|} ).
\end{equation}
On the other hand, from  \eqref{eq:expression_error} we deduce that in $B(0,R_\e)$ there holds
\begin{equation*}\begin{split}
&S_a(V_d)\overline{w}_{x_1}  = V_d\left\{ \left(1-\big|\prod_{j=1}^n w^j \big|^2 \right)-\sum_{j=1}^n \left(1-|w^j|^2 \right)+\sum_{j=1}^n \sum_{l\neq j} \frac{\nabla w^l}{w^l}\frac{\nabla w^j}{w^j} \right\}\overline{w}^1_{x_1} \nonumber \\
& =2V_d\sum_{j\neq1} \frac{\nabla w^1}{w^1}\frac{\nabla w^j}{w^j} \overline{w}^1_{x_1}+O(\e^2|\log\e|)  \nonumber \\
& = 2V_d\sum_{j\neq 1} \left[\left( \frac{\rho'_1 \rho'_j}{\rho_1 \rho_j}-\frac{1}{r_1 r_j} \right) \cos(\theta_1-\theta_j)+i\left(\frac{\rho_1'}{\rho_1r_j}-\frac{\rho_j'}{\rho_j r_1} \right) \sin(\theta_1-\theta_j) \right] \left( \rho_1'\cos \theta_1+i\frac{\rho_1\sin \theta_1}{r_1} \right) e^{-i\theta_1}\\
&\quad +O(\e^2 |\log \e|). \nonumber
\end{split}\end{equation*}
Thus we find
\begin{equation*}\begin{split}
B_a &= -2\sum_{j\neq1} \int_{\{r_1<R_\e\}}\frac{\rho_1'\rho_1}{r_1 r_j}\left( \cos(\theta_1-\theta_j)\cos \theta_1+\sin(\theta_1-\theta_j)\sin \theta_1 \right) +O(\e^2 |\log \e|)\nonumber \\
&= -2\sum_{j\neq1}  \int_{\{r_1<R_\e\}}\frac{\rho_1' \rho_1 \cos\theta_j}{r_1r_j} +O(\e^2 |\log \e|). \nonumber
\end{split}\end{equation*}
To compute the last term we observe that, since \( R_\e=o_\e\left( \frac{1}{\e \sqrt{|\log \e|}} \right)\), inside the ball \( \{r_1<R_\e\}\) we have
\begin{equation*}\begin{split}
 \frac{\cos \theta_j}{r_j}&=\frac{|\RE(\xi_j^+-\xi_1^+)|}{|\xi_j^+-\xi_1^+|^2}(1+o_\e(1)) \quad \text{ for every } j \neq 1,\\
 &=\frac{d_\e\left(1-\cos\left(2\pi(j-1)/n\right)\right)}{2d_\e^2\left(1-\cos\left(2\pi(j-1)/n\right)\right)}(1+o_\e(1)) \\
 &=\frac{1}{2d_\e}(1+o_\e(1)) ,
\end{split}\end{equation*}
where we have used that \(\xi_j^+=d_\e e^{2i(j-1)\pi/n}\). Hence,
\begin{align*}
\sum_{j=2}^n \int_{r_1<R_\e} \frac{\cos \theta_j}{r_j} \rho'_1\rho_1 dr_1d \theta_1 &=\frac{2\pi}{2d_\e} (n-1)(1+o_\e(1))\int_0^{R_\e} \rho_1'\rho_1 \dif r_1 .
\end{align*}
Noticing that \( \int_0^{R_\e} \rho_1'\rho_1 \dif r_1 =\frac12 \left( \rho_1^2(R_\e)-\rho_1(0) \right)=\frac12 +o_\e(1)\) we conclude
\begin{align*}
\sum_{j=2}^n \int_{r_1<R_\e} \frac{\cos \theta_j}{r_j} \rho'_1\rho_1 \dif r_1 \dif  \theta_1 
&= \frac{\pi}{d_\e}\frac{n-1}{2}(1+o_\e(1)),
\end{align*}
and thus
\begin{equation}
B_a=-\frac{n-1}{\d_\e}\e\sqrt{|\log \e|}\pi+o_\e(\e\sqrt{|\log \e|}).
\end{equation}
For the last term in the error we have
\begin{equation*}\begin{split}
&R_c(z)\overline{w}_{x_1}(z-\xi_1^+)\\
&\quad =c\hat{d}_\e\e\sqrt{|\log\e|}\sum_{j=1}^n \left(  \sin(\theta_j-\varphi_j)\frac{\rho'_j}{\rho_j}+i\frac{\cos(\theta_j-\varphi_j)}{r_j}\right)\left(\rho'_1\cos\theta_1+i\frac{\rho_1 \sin\theta_1}{r_1}\right)e^{-i\theta_1}.
\end{split}\end{equation*}
Thus, by using that \(\varphi_1=0\), we find
\begin{equation*}\begin{split}
B_c&=-c\hat{d}_\e\e\sqrt{|\log\e|}\int_{|\rho_1|< R_\e}\frac{\rho'_1\rho_1}{r_1}\left(\sin^2\theta_1+\cos^2\theta_1 \right)+O\left( \e^2|\log \e|^{3/2} \right) \nonumber \\
&=-c\hat{d_\e}\e\sqrt{|\log\e|}2\pi\int_0^{R_\e}\rho'_1(r_1) \rho_1(r_1)\,\dif  r_1+O\left(\e^2|\log \e|^{3/2} \right). \nonumber 
\end{split}\end{equation*}
Therefore, since 
$$\int_0^{R_\e}\rho'_1(r_1) \rho_1(r_1)\, \dif r_1=\frac12 (\rho^2(R_\e)-\rho^2(0))=\frac 12+O(\e^2|\log\e|),$$
we find
\begin{equation}
B_c=-c\hat{d_\e}\e\sqrt{|\log\e|}\pi +o_\e(\e\sqrt{|\log\e|}).
\end{equation}

Hence, we conclude that
\begin{equation*}
 c_1 c_*=\e \sqrt{|\log \e|}\left( -\frac{a_0}{\d_\e}+\tilde{a}_1\d_\e\right)+o_\e(\e \sqrt{|\log \e|}),
\end{equation*}
with
\begin{equation}\label{a1}
a_0:=(n-1)\pi,\quad  \tilde{a}_1:= (1-c)\pi,
\end{equation}
which is positive since we assumed $c<1$.
Let us point out that in this expression $o_\e(\e \sqrt{|\log \e|})$ is a continuous function of the parameter $\d$. 

By applying the intermediate value theorem we can find $\d_0$ near $\sqrt{\frac{a_0}{\tilde{a}_1}}=\sqrt{\frac{n-1}{1-c}}$ such that 
$$c_1=c_1(\d_0)=0,$$
and therefore, for such $\d_0$ we conclude that $V_d+\varphi_\e$ is a solution of \eqref{eq:GP2Drescaled}.

\medskip

{\it Proof of Theorem \ref{th:main2}.} Assume $n^+=n+1$ and $n^-=1$ in \eqref{eq:approx}, with $\xi_1^-=0$. The result follows analogously to the case of Theorem \ref{th:main1}. Indeed, estimates \eqref{projL} and \eqref{projN} hold straightforward, so we only have to estimate the projection of the error term. Let us define $B_a$, $B_b$ and $B_c$ as in \eqref{defB}. \eqref{defB2}. Since $|\xi_1^-|=0$ the terms $B_b$ and $B_c$ are estimated exactly as in the proof of Theorem \ref{th:main1}, to get
$$B_b=\d_\e\pi \e \sqrt{|\log \e|} +o_\e(\e \sqrt{| \log \e|} ),\qquad B_c=-c\hat{d}\e\sqrt{|\log\e|}\pi +o_\e(\e\sqrt{|\log\e|}).$$ 
 To estimate $B_a$ we see that in this case
\begin{equation*}\begin{split}
&S_a(V_d)\overline{w}_{x_1}   \\
& = 2V_d\sum_{j\neq 1} \left[\left( \frac{\rho'_1 \rho'_j}{\rho_1 \rho_j}-\frac{1}{r_1 r_j} \right) \cos(\theta_1-\theta_j)+i\left(\frac{\rho_1'}{\rho_1r_j}-\frac{\rho_j'}{\rho_j r_1} \right) \sin(\theta_1-\theta_j) \right] \left( \rho_1'\cos \theta_1+i\frac{\rho_1\sin \theta_1}{r_1} \right)e^{-i\theta_1} \\
&\qquad +2V_d\left[\left( \frac{\rho'_1 \rho'}{\rho_1 \rho}+\frac{1}{r_1 r} \right) \cos(\theta_1-\theta)+i\left(\frac{\rho_1'}{\rho_1 r}+\frac{\rho'}{\rho r_1} \right) \sin(\theta-\theta_1) \right] \left( \rho_1'\cos \theta_1+i\frac{\rho_1\sin \theta_1}{r_1} \right)e^{-i\theta_1} \\
&\qquad +O\left(\e^2 |\log \e|\right),
\end{split}\end{equation*}
and thus
\begin{equation*}\begin{split}
B_a &= -2\sum_{j\neq1} \int_{\{r_1<R_\e\}}\frac{\rho_1'\rho_1}{r_1 r_j}\left( \cos(\theta_1-\theta_j)\cos \theta_1+\sin(\theta_1-\theta_j)\sin \theta_1 \right) \\
&\quad +2 \int_{\{r_1<R_\e\}}\frac{\rho_1'\rho_1}{r_1 r}\left( \cos(\theta_1-\theta)\cos \theta_1+\sin(\theta_1-\theta)\sin \theta_1 \right) +O(\e^2 |\log \e|)\nonumber \\
&= -2\sum_{j\neq1}  \int_{\{r_1<R_\e\}}\frac{\rho_1' \rho_1 \cos\theta_j}{r_1r_j} +2\int_{\{r_1<R_\e\}}\frac{\rho_1' \rho_1 \cos\theta}{r_1r} +O(\e^2 |\log \e|), \nonumber
\end{split}\end{equation*}
what implies
\begin{equation*}\begin{split}
B_a &= -2\sum_{j\neq 1} \int_{\{r_1<R_\e\}}  \frac{\rho_1' \rho_1 \cos\theta_j}{r_j} \dif r_1  \dif \theta_1 +2\int_{\{r_1<R_\e\}}  \frac{\rho_1' \rho_1 \cos\theta}{r} \dif  r_1 \dif \theta_1 +O(\e^2 |\log \e|)\nonumber \\
&= -\e \sqrt{|\log \e|}\frac{\tilde{a}_0}{\hat{d}_\e}+o_\e(\e \sqrt{|\log \e|}),
\end{split}\end{equation*}
with \begin{equation}
\tilde{a}_0:=n_+-3.
\end{equation} 
Therefore, 
\begin{equation*}
 c_1 c_*=\e \sqrt{|\log \e|}\left( -\frac{\tilde{a}_0}{\d_\e}+a_1\d_\e\right)+o_\e(\e \sqrt{|\log \e|}),
\end{equation*}
with $a_1$ defined in \eqref{a1}, and the result follows as in the previous case. More precisely, thanks to the intermediate value theorem we find \(\d_\e\) near \( \sqrt{\frac{\tilde{a}_0}{a_1}}=\sqrt{\frac{n_+-3}{1-c}}\) such that \(c_1=0\).

\section*{Acknowledgements}
J.~D\'avila has been supported  by  a Royal Society  Wolfson Fellowship, UK.
 M.~del Pino has been supported by a Royal Society Research Professorship, UK. M.~Medina  has been partially supported by Project PDI2019-110712GB-100, MICINN, Spain. R.~Rodiac has been partially supported by the ANR project  BLADE Jr. ANR-18-CE40-0023.
\section*{Appendix}
\subsection{The standard vortex and its linearized operator}
The building block used to construct our solutions to equation \eqref{eq:GP2Drescaled} is the standard vortex of degree one in $\R^2$, that we denote $w$. It satisfies
$$\Delta w+(1-|w|^2)w=0\qquad \mbox{ in }\R^2,$$
and can be written as
$$w(x_1,x_2)= \rho(r)e^{i\v}\mbox{ where }x_1=r \cos \v,\; x_2=r \sin \v.$$
Here $\rho$ is the unique solution of 
\begin{equation}\label{eq:equation_modulus_w}  \left\{  \begin{aligned}
&\rho'' +  \frac {\rho'}r  - \frac \rho{r^2} +  (1-\rho^2) \rho  =  0 \inn (0,\infty),\\
& \rho(0^+) = 0 , \quad \rho(+\infty) = 1 ,
\end{aligned} \right. 
\end{equation} see \cite{ChenElliottQi1994,HerveHerve1994}. 
In this section we collect useful properties of $\r$.
\begin{lemma}\label{lem:propertiesofrho}
Let $\rho$ be the unique solution of \eqref{eq:equation_modulus_w}. Then:
\begin{itemize}
\item[1)] $\r(0)=0$, $\r'(0)>0$, $0<\r(r)<1$ and $\r'(r)>0$ for all $r>0$,
\item[2)] $\r(r)=1-\frac{1}{2r^2}+O(\frac{1}{r^4})$ for large $r$,
\item[3)] $\r(r)=\alpha r-\frac{\alpha r^3}{8}+O(r^5)$ for $r$ close to $0$ for some $\alpha>0$,
\item[4)] if we define $T(r):=\r'(r)-\frac{\r}{r}$ then $T(0)=0$ and $T(r)<0$ in $(0,+\infty)$,
\item[5)] $\r'(r)=\frac{1}{r^3}+O(\frac{1}{r^4})$, $\r''(r)=O(\frac{1}{r^4})$ for large $r$.
\end{itemize}
\end{lemma}
For the proof of this lemma we refer to \cite{HerveHerve1994, ChenElliottQi1994}.

An object of special importance to construct our solution is the linearized Ginzburg-Landau operator around $w$, defined by
\begin{equation}\nonumber
L(\phi):= \Delta \phi+(1-|w|^2)\phi-2\RE(\overline{w}\phi)w.
\end{equation}
This operator does have a kernel, as the following result states.

\begin{lemma}\label{lem:ellipticestimatesL0}
Suppose that $\phi \in L^\infty(\R^2)$ satisfies $L(\phi)=0$ in $\R^2$ and the symmetry  \(\phi(\overline{z})=\overline{\phi}(z)\).
Assume furthermore that when we write $\phi=iw\psi$ and $\psi=\psi_1+i\psi_2$ with $\psi_1,\psi_2\in \R$ we have
\begin{eqnarray*}
|\psi_1|+(1+|z|)|\nabla \psi_1| \leq  C, \qquad |\psi_2|+|\nabla \psi_2| \leq  \frac{C}{1+|z|} ,\quad |z|>1.
\end{eqnarray*}
Then
\begin{equation*}
\phi=c_1 w_{x_1}
\end{equation*}
for some real constant \(c_1\).
\end{lemma}

\begin{lemma}
\label{lem:ellipticestimatesL0-b}
Suppose that $\phi \in L_{\text{loc}}^\infty(\R^2)$ satisfies $L(\phi)=0$ in $\R^2$ and the symmetry  \(\phi(\overline{z})=\overline{\phi(z)}\).
Assume furthermore that when we write $\phi=iw\psi$ and $\psi=\psi_1+i\psi_2$ with $\psi_1,\psi_2\in \R$ we have
\begin{align*}
|\psi_1|+(1+|z|)|\nabla \psi_1| \leq  C(1+|z|)^\alpha, \qquad |\psi_2|+|\nabla \psi_2| \leq  \frac{C}{1+|z|} ,
\quad |z|>1,
\end{align*}
for some $\alpha<3$.
Then
\begin{equation*}
\phi=c_1 w_{x_1}
\end{equation*}
for some real constant \(c_1\).
\end{lemma}

The proofs of these results can be found in \cite[Lemma 7.1 and Lemma 7.2]{DdPMR}.

\bibliographystyle{abbrv}
\bibliography{biblio}

\end{document}